\documentclass[11pt,oneside,reqno]{amsart}  

\usepackage[utf8]{inputenc}            
\usepackage[T1]{fontenc}
\usepackage{microtype}

\usepackage[english]{babel}

\usepackage{url}
\usepackage{xcolor}                   
\usepackage{graphics,graphicx,xypic}
\usepackage{dsfont,txfonts}
\usepackage{mathrsfs}
\usepackage{textcomp}
\usepackage{enumerate} 
\usepackage{hyperref}
\usepackage[toc,page]{appendix}         

\usepackage[a4paper,scale={0.72,0.74},marginratio={1:1},footskip=7mm,headsep=10mm]{geometry}

\numberwithin{equation}{section}    

\theoremstyle{plain}
\newtheorem{theorem}{Theorem}[section]

\theoremstyle{plain}
\newtheorem{assumption}[theorem]{Assumption}

\theoremstyle{plain}
\newtheorem{proposition}[theorem]{Proposition}
\newtheorem{corollary}[theorem]{Corollary}
\newtheorem{lemma}[theorem]{Lemma}

\theoremstyle{definition}
\newtheorem{definition}[theorem]{Definition}
\newtheorem{remark}[theorem]{Remark}

\newcommand{\X}{\mathrm{X}}
\newcommand{\Xden}{\mathrm{X}^{(\mathtt{fin})}}
\newcommand{\Xdi}{\mathrm{X}^{(N)}}
\DeclareMathOperator*{\argmax}{arg\,max}

\newcommand{\tolaw}{\overset{(\mathrm{d})}{\to}}

\newcommand{\Rnk}{\rho_{N}^{(k)}}
\newcommand{\hb}{\hat{\beta}}

\newcommand{\sD}{\hat{\sigma}_{N}}
\newcommand{\sC}{\hat{\sigma}_{\infty}}
\newcommand{\sDtr}{\hat{\sigma}_{ N}^{(k)}}
\newcommand{\sCtr}{\hat{\sigma}_{\infty}^{(k)}}

\renewcommand{\Uc}[1]{U_{{#1}, \infty}}
\newcommand{\Udtr}[1]{U_{{#1}, N}^{(k)}}
\newcommand{\Uctr}[1]{U_{{#1},\infty}^{(k)}}

\newcommand{\Id}[1]{\hat{I}_{{#1}, N}}
\newcommand{\Ic}[1]{\hat{I}_{{#1}, \infty}}
\newcommand{\Idtr}[1]{\hat{I}_{{#1}, N}^{(k)}}
\newcommand{\Ictr}[1]{\hat{I}_{{#1},\infty}^{(k)}}

\newcommand{\Idp}[1]{{I}_{{#1}, N}}

\newcommand{\hbc}{\hat{\beta}_{\texttt{c}}}

\newcommand{\ud}[1]{\hat{u}_{{#1}, N}}
\newcommand{\uc}[1]{\hat{u}_{{#1}, \infty}}
\newcommand{\udtr}[1]{\hat{u}_{{#1}, N}^{(k)}}
\newcommand{\uctr}[1]{\hat{u}_{{#1},\infty}^{(k)}}

\newcommand{\Gib}{{\hat{\mathrm{P}}}_{{\hb_N},N}}

\newcommand{\ind}{\mathds{1}}

\newcommand\blfootnote[1]{%
\begin{minipage}[b]{1 \textwidth}
{\small {#1}}
\end{minipage}
}

\title{Pinning Model with Heavy-Tailed disorder}

\author{Niccolò Torri}

\address{Universit\'e de Lyon, Institut Camille Jordan,
  Universit\'e Claude Bernard - Lyon 1, 43 bd du 11 novembre 1918, 69622 Villeurbanne cedex, France
 }

\email{torri@math.univ-lyon1.fr}

\keywords{
 Pinning Model; Directed Polymers; Heavy Tails; Localization.}

\subjclass[2010]{60G57; 60G55; 82B44; 82D60} 

\thanks{
This work was supported by the \emph{Programme Avenir Lyon Saint-Etienne de l'Université de Lyon} (ANR-11-IDEX-0007), within the program \emph{Investissements d'Avenir} operated by the French National Research Agency (ANR)\\
\vskip1pt
\blfootnote{
This is an electronic reprint of the original article published by 
\emph{Stochastic Processes and their Applications} (2015), \url{http://dx.doi.org/10.1016/j.spa.2015.09.010}
. 
}}

\begin{document}

\begin{abstract}
We study the pinning model, which describes the behavior of a Markov chain interacting with a distinguished state. The interaction depends on an external source of randomness, called disorder. Inspired by \cite{AL11} and \cite{HM07}, we consider the case when the disorder is heavy-tailed, while the return times of the Markov chain are stretched-exponential. We prove that the set of times at which the Markov chain visits the distinguished state, suitably rescaled, has a limit in distribution. Moreover there exists a random threshold below which this limit is trivial. Finally we complete a result of \cite{AL11} on the directed polymer in a random environment.
\end{abstract}

\maketitle

\section{Set-up and Results}
The pinning model can be defined as a random perturbation of a random walk or, more generally, of a Markov chain called $S$. In this model we modify the law of the Markov chain by weighing randomly the probability of a given trajectory up to time $N$. 
Each time $S$ touches a distinguished state, called $0$, before $N$, say at time $n$, we give a reward or a penalty to this contact by assigning an exponential weight $\exp(\beta\omega_n-h)$, where $\beta\in\mathbb{R}_+:=(0,\infty)$, $h\in\mathbb{R}$ and $(\omega=(\omega_n)_{n\in\mathbb{N}},\mathbb{P})$ is an independent random sequence called disorder.
 The precise definition of the model is given below.

\smallskip

In this model we perturb $S$ only when it takes value $0$, therefore it is convenient to work 
 with its zero level set. For this purpose we consider a renewal process $(\tau=(\tau_n)_{n\in\mathbb{N}},\mathrm{P})$, that is an  $\mathbb{N}_0$-valued random process such that $\tau_0=0$ and  $(\tau_j-\tau_{j-1})_{j\in\mathbb{N}}$ is an  i.i.d. sequence. This type of random process can be thought of as a random subset of $\mathbb{N}_0$, in particular if $S_0=0$, then by setting $\tau_0=0$ and $\tau_j=\inf\{k>\tau_{j-1} : S_k=0\}$, for $j>0$, we recover the zero level set of the Markov chain $S$. From this point of view the notation $\{n\in\tau\}$ means that there exists $j\in\mathbb{N}$ such that $\tau_j=n$. We refer to \cite{A03,GB07} for more details about the theory of the renewal processes.
 
 \smallskip

In the literature, e.g. \cite{FdH07,GB10,GB07}, typically the law of $\tau_1$, the inter-arrival law of the renewal process, has a polynomial tail and the disorder has finite exponential moments. In our paper we study the case in which the disorder has polynomial tails, in analogy with the articles 
\cite{AL11} and 
\cite{HM07}. To get interesting results we work with a renewal process where the law of $\tau_1$ is stretched-exponential (cf. Assumptions \ref{ASSrenP}). Possible generalizations will be discussed in Section \ref{SecPersp}.

\subsection{The Pinning Model}
In this paper we want to understand the behavior of $\tau/N\cap[0,1]=\{\tau_j/N : \tau_j \leq N\}$, the rescaled renewal process up to time $N$,when $N$ gets large.

\medskip

  We denote by $\mathrm{P}_N$ the law of $\tau/N\cap[0,1]$, which turns out to be a probability measure on the space of all subsets of $\{0,1/N,\cdots,1\}$.
On this space, for $\beta,h\in\mathbb{R}$ we define the \emph{pinning model} $\mathrm{P}_{\beta,h,N}^{{\omega}}$ as a probability measure defined by the following Radon-Nikodym derivative
\begin{equation}\label{eqIntr1}
\frac{\textrm{d}\mathrm{P}_{\beta,h,N}^{{\omega}}}{\textrm{d}\mathrm{P}_N}(I)=
\frac{1}{\mathrm{Z}_{\beta,h,N}^{\omega}}\exp\left(\displaystyle\sum_{n=1}^{N-1} (\beta\omega_n-h)\ind(n/N\in I)\right)\ind(1\in I),
\end{equation}
where $\mathrm{Z}_{\beta,h,N}^{\omega}$ is a normalization constant, called partition function, that makes $\mathrm{P}_{\beta,h,N}^{\omega}$ a probability. 
Let us stress that a realization of $\tau/N\cap [0,1]$ has non-zero probability only if its last point is equal to $1$. This is due to the presence of the term $\ind(1\in I)$ in \eqref{eqIntr1}.
In such a way the pinning model is a \emph{random} probability measure on the space $\X$  of all closed subsets of $[0,1]$ which contain both $0$ and $1$ 
\begin{equation}
\X=\{I\subset [0,1]: I \textrm{ is closed and }0,1\in I\}
\end{equation}
with support given by $\Xdi$, the set of all subsets of $\{0,1/N,\cdots,1\}$ which contains both $0$ and $1$. 

\medskip

The pinning model $\mathrm{P}_{\beta,h,N}^{{\omega}}$ is a \emph{random} probability measure, in the sense that it depends on a parameter $\omega$, called disorder, which is a quenched realization of a random sequence. Therefore in the pinning model we have two (independent) sources of randomness: the renewal process $(\tau,\mathrm{P})$ and the disorder $(\omega,\mathbb{P})$. To complete the definition we thus need to specify our assumptions about the disorder and the renewal process.

\begin{assumption}\label{ASSdis} 
We assume that the disorder $\omega$ is an i.i.d. sequence of random variables whose tail is regularly varying with index $\alpha\in (0,1)$, namely
\begin{equation}\label{eqIntr2}
\mathbb{P}(\omega_1>t)\sim L_0(t)t^{-\alpha}, \quad t\to \infty,
\end{equation}
where $\alpha\in (0,1)$ and $L_0(\cdot)$ is a slowly varying function, cf. \cite{BGT89}. 
Moreover we assume that the law of $\omega_1$ has no atom and it is supported in $(0,\infty)$, i.e. $\omega_1$ is a positive random variable. The reference example to consider is given by the Pareto Distribution.
\end{assumption}

\begin{assumption}\label{ASSrenP} 
Given a renewal process, we denote the law of its first point $\tau_1$ by $K(n):=\mathrm{P}(\tau_1=n)$, which characterizes completely the process. 
Throughout the paper we consider a non-terminating renewal process $\tau$, i.e., $\sum_{n\in\mathbb{N}}K(n)=1$, which satisfies the two following assumptions
\begin{enumerate}
 \item \label{A1} Subexponential, cf. 
 \ref{Aa1}: 
 \begin{align}\nonumber
 \forall\, k>0,\, \lim_{n\to\infty}K(n+k)/K(n)=1 \quad \text{and}\quad \lim_{n\to\infty}K^{*(2)}(n)/K(n)=2,
 \end{align}
 
 \item \label{A2} Stretched-exponential
 \begin{align}\nonumber
\exists\, \gamma\in (0,1), \, \mathtt{c}>0\,:\quad
\lim_{n\to\infty} \log K(n)/ n^{\gamma} = -\mathtt{c} 
\end{align}
\end{enumerate}
\end{assumption}

\begin{remark}
Roughly speaking, up to local regularity assumptions (subexponentiality), we take $K(n)\cong e^{-\mathtt{c}n^{\gamma}}$. More precisely these conditions are satisfied if
\begin{equation}\label{eqIntr6}
 K(n)\sim \frac{L(n)}{n^{\rho}}e^{-\mathtt{c} n^{\gamma}}, \quad n\to\infty,
\end{equation}
with $\rho\in \mathbb{R}$ and $L(\cdot)$ a slowly varying function, cf. Section \ref{Aa1}. 
\end{remark}

\subsection{Main Results}
The aim of this paper is to study the behavior of $\tau/N\cap [0,1]$ under the probability $\mathrm{P}_{\beta,h,N}^{{\omega}}$, when $N$ gets large. 
To have a non trivial behavior we need to fix $h>0$ (which is actually equivalent to set $h=0$ in (\ref{eqIntr1}) and consider a terminating renewal process, cf. Section \ref{Sech3}) and send $\beta$ to $0$ as $N\to\infty$. If $\beta$ goes to $0$ too slowly (or if it does not go to $0$ at all), then $\tau/N\cap [0,1]$ will always converge to the whole $[0,1]$, if it goes too fast, it will converge to $\{0,1\}$. The interesting regime is the following:
\begin{equation}\label{eqIntr7bis}
\beta_N\sim\hb N^{\gamma-\frac{1}{\alpha}}\ell(N), \quad N\to\infty,
\end{equation} 
with $\ell$ a particular slowly varying function defined by $L_0$ in (\ref{eqIntr2}).
Under such rescaling of $\beta$ and such choice of $h>0$ we prove the existence of a random threshold $\hbc$: if $\hb<\hbc$ then $\tau/N\cap [0,1]$ converges to $\{0,1\}$, while if $\hb>\hbc$ then its limit has at least one point in $(0,1)$. 

\medskip

To prove these facts we proceed by steps. In the first one we show that there exists a random set around which $\tau/N\cap [0,1]$ is concentrated with respect to the Hausdorff distance: given two non-empty sets $A,B\subset [0,1]$ 
\begin{equation}\label{eqHaus11}
 d_H(A,B)=\max\left\{\sup_{a\in A}d(a,B),\sup_{b\in B}d(b,A)\right\},
\end{equation}
where $d(z,C)=\inf_{c\in C} |z-c|$ is the usual distance between a point and a set.

\begin{theorem}\label{TheoremMain2} Let $(\beta_N)_N$ be as in \eqref{eqIntr7bis}. Then for any  $N\in \mathbb{N},\, \hb>0$ there exists a random set $ \Idp{\beta_N}$  (i.e. an $\X$-valued random variable) such that for any $\delta,h>0$ one has that $\mathrm{P}_{{\beta}_N,h,N}^{{\omega}}\left(d_H(I,\Idp{\beta_N})>\delta\right)$ converges to $0$ as $N\to\infty$ in probability (with respect to the disorder $\omega$). More precisely for any $\varepsilon>0$ there exist $\nu=\nu(\varepsilon,\delta)$ and $\hat{N}$ such that for all $N>\hat{N}$
 \begin{equation}
\mathbb{P}\left(\mathrm{P}_{{\beta}_N,h,N}^{{\omega}}\left(d_H(I,\Idp{\beta_N})>\delta\right)<e^{-\nu N^{\gamma}}\right)>1-\varepsilon.
 \end{equation}
\end{theorem}

The second step regards the convergence in law of $\Idp{\beta_N}$.

\begin{theorem}\label{TheoremMain1}  Let $(\beta_N)_N$ be  as in \eqref{eqIntr7bis}.
Then for any $\hat{\beta}>0$ there exists a random closed subset $\Ic{\hb}\in \X$ (i.e. an $\X$-valued random variable), which depends on a suitable continuum disorder (defined in section \ref{secDisorder}), such that
\begin{equation}
\Idp{\beta_N}\tolaw\Ic{\hb}, \quad N\to\infty
\end{equation}
on $(\X,d_H)$.
\end{theorem}

As a consequence of these Theorems, if we look at $ \mathrm{P}_{{\beta}_N,h,N}^{{\omega}}$ as a random probability on $\X$, i.e. as a random variable in $\mathcal{M}_1(\X,d_H)$, the space of the probability measures on $\X$, then Theorems \ref{TheoremMain2} and \ref{TheoremMain1} imply that it converges in law to the $\delta$-measure concentrated on the limit set $\Ic{\hb}$. 

\begin{theorem}\label{TheoremMain3}  Let $(\beta_N)_N$ be as in \eqref{eqIntr7bis}. Then for any $h,\hat{\beta}\in (0,\infty)$,
 \begin{equation}
   \mathrm{P}_{{\beta}_N,h,N}^{{\omega}}\tolaw\delta_{\Ic{\hb}}, \quad N\to\infty
 \end{equation}
  on $\mathcal{M}_1(\X,d_H)$ equipped with the weak topology.
\end{theorem}

This concludes our results about the convergence of the random set $\tau/N\cap [0,1]$, now we want to discuss the structure of its limit. We prove that there exists a critical point $\hbc$ such that, if $\beta<\hbc$, then $\tau/N\cap [0,1]$ has a trivial limit, given by $\{0,1\}$. Otherwise, if $\beta>\hbc$, then the limit sets has points in $(0,1)$.

\smallskip

We define the random threshold $\hbc$ as

\begin{equation}\label{EqRPT}
\hbc=\inf\{\hb : \Ic{\hb}\not\equiv \{0,1\}\}.
\end{equation}

Denoting by $\mathbb{P}$ the law of the continuum disorder, by a monotonicity argument (cf. Section \ref{StrI}) we have that

\begin{enumerate}
\item \label{a} If $\hb<\hbc$, then $\Ic{\hb}\equiv\{0,1\}$, $\mathbb{P}$-a.s.
\item \label{b} If $\hb>\hbc$, then $\Ic{\hb}\not\equiv\{0,1\}$, $\mathbb{P}$-a.s.
\end{enumerate} 
Moreover the structure of $\hbc$ is described by the following result
\begin{theorem}\label{TheoremMain4}  For any choice of $\alpha,\gamma\in (0,1)$ we have that
$\hbc>0$ $\mathbb{P}$-a.s., where $\alpha$ is the disorder exponent in  Assumption \ref{ASSdis}, while $\gamma$ is the renewal exponent of Assumption \ref{ASSrenP}.
\end{theorem}

By using the same technique we complete the result \cite[Prop. 2.5]{AL11} about the structure of $\beta_c$, the random threshold defined for the directed polymer model in a random environment with heavy tails (we recall its definition in Section \ref{SecGamma}). Precisely
\begin{theorem}\label{eq32al}
Let $\beta_c$ as in \eqref{ubetacAA}, then, if $\mathbb{P}_{\infty}$ denotes the law of the continuum environment,
 \begin{enumerate}
  \item For any $\alpha\in (0,\frac{1}{2})$, $\beta_c>0$, $\mathbb{P}_{\infty}$-a.s.
  \item For any $\alpha\in [\frac{1}{2},2)$, $\beta_c=0$, $\mathbb{P}_{\infty}$-a.s.
 \end{enumerate}
\end{theorem}
\begin{remark}\label{rem:AL11bcora}
In \cite{AL11} the value of $\beta_c$ was unknown for $\alpha\in (1/3, 1/2)$.
\end{remark}

\subsection{Organization of the Paper}
In rest of the paper  we prove the results of this section. Section \ref{SecEEVE} contains some preliminary definitions and tools that we use for our proofs. Sections \ref{SecConver} contains the proof of Theorem \ref{TheoremMain1} and Section \ref{SecConcen} the proof of Theorems \ref{TheoremMain2} and \ref{TheoremMain3}. In Section \ref{StrI} we prove Theorem \ref{TheoremMain4} and then in Section \ref{SecGamma} we recall the definition of the Directed Polymer Model, proving Theorem \ref{eq32al}. 
Finally in Section \ref{SecPersp} we discuss the choice of the parameters $\alpha,\gamma$ and the future perspectives.

\section{Energy \& entropy}\label{SecEEVE}

In this section we define the random sets $\Idp{\beta}$, $\Ic{\hb}$ and we motivate the choice of $\beta_N$ in (\ref{eqIntr7bis}).

\medskip

To define the random set $\Idp{\beta}$ we compare the \emph{Energy} and the \emph{entropy} of a given configuration: for a finite set $I=\{x_0=0<x_1<\cdots<x_{\ell}=1\}$ we define its \emph{Energy} as
\begin{equation}\label{eqIntr3}
 \sigma_{N}(I) = \sum_{n=1}^{N-1} \omega_n\mathds{1}(n/N\in I)
\end{equation}
and its \emph{entropy} as 
\begin{equation}
 E(I)=\sum_{k=1}^{\ell}(x_i-x_{i-1})^{\gamma}.
\end{equation}

By using these two ingredients we define
\begin{equation}\label{eqIntr5}
 \Idp{\beta}=\displaystyle\argmax_{I\in\mathrm{X}^{(N)}}\left( \beta\sigma_N(I)-\mathtt{c}N^{\gamma}E(I)\right),
\end{equation}
where $\gamma $ and $\mathtt{c}$ are defined in \eqref{A2} of Assumption \ref{ASSrenP} and $\Xdi$ is the space of all possible subsets of $\{0,1/N,\cdots,1\}$ containing $0$ and $1$. 

\smallskip

By using (\ref{eqIntr5}) we can find the right rescaling for $\beta$: indeed it  has to be chosen in such a way to make the Energy and the entropy comparable. For this purpose it is convenient to work with a rescaled version of the disorder. We consider $(\tilde{M}_i^{(N)})_{i=1}^{N-1}$ the ordered statistics of $(\omega_i)_{i=1}^{N-1}$ --- which means that $\tilde{M}_1^{(N)}$ is the biggest value among $\omega_1,\cdots, \omega_{N-1}$, $\tilde{M}_2^{(N)}$ is the second biggest one and so on --- and $(Y_i^{(N)})_{i=1}^{N-1}$ a random permutation of $\{\frac{1}{N},\cdots 1-\frac{1}{N}\}$, independent of the ordered statistics. The sequence $((\tilde{M}_i^{(N)}, Y_i^{(N)})_{i=1}^{N-1}$ recovers the disorder $(\omega_i)_{i=1}^{N-1}$. 
The asymptotic behavior of such sequence is known and it allows us to get the right rescaling of $\beta$. Let us recall the main result that we need.

\subsection{The Disorder}\label{secDisorder}
Let us start to note that for any fixed $k$ as $N\to \infty$
\begin{equation}\label{EqConvR1.1}
 (Y_i^{(N)})_{i=1,\cdots,k}\tolaw (Y_i^{(\infty)})_{n=1,\cdots,k},
\end{equation}
where $(Y_i^{(\infty)})_{i\in\mathbb{N}}$ is an i.i.d. sequence of $\textrm{Uniform}([0,1])$. 

\medskip

For the ordered statistics, from classical extreme value theory, see e.g. \cite[Section 1.1]{R87}, we have that there exists a sequence $(b_N)_N$ such that for any fixed $k>0$, as $N\to \infty$
\begin{equation}\label{EqConvR1}
 (M_i^{(N)}:=b_N^{-1}\tilde{M}_i^{(N)})_{i=1,\cdots,k}\tolaw (M_i^{(\infty)})_{n=1,\cdots,k},
\end{equation}
where $M_i^{(\infty)}=T_i^{-{1}/{\alpha}}$, with $T_i$ a sum of $i$ independent exponentials of mean $1$ and $\alpha$ is the exponent of the disorder introduced in (\ref{eqIntr2}). 
The sequence $b_N$ is characterized by the following relation
\begin{equation}\label{eq:b_Nbev}
\mathbb{P}\left(\omega_1>b_N\right)\sim\frac{1}{N}, \quad N\to\infty
.
\end{equation}
This implies that $b_N\sim N^{\frac{1}{\alpha}}\ell_0(N)$, where $\ell_0(\cdot)$ is a suitable slowly varying function uniquely defined by $L_0(\cdot)$, cf. (\ref{eqIntr2}). 

\smallskip

We can get a stronger result without a big effort, which will be very useful in the sequel. Let us consider the (independent) sequences $({M}_i^{(N)})_{i=1}^{N-1}$ and $(Y_i^{(N)})_{i=1}^{N-1}$ and 
\begin{align}
  &w_i^{(N)} := \begin{cases}
  ({M}_i^{(N)},Y_i^{(N)})_{i=1}^{N-1}, & i<N, \\ 
  0, & {i\geq N},\end{cases} \\
  & w_i^{(\infty)} :=  ({M}_i^{(\infty)},Y_i^{(\infty)})_{i\in\mathbb{N}},
 \end{align}
We can look at $w^{(N)}=(w_i^{(N)})_{i\in\mathbb{N}}$ and $w^{(\infty)}=(w_i^{(\infty)})_{i\in\mathbb{N}}$ as random variables taking values in $\mathcal{S} := (\mathbb{R}^2)^{\mathbb{N}}$.
Let us equip  $\mathcal{S}$ with the product topology: a sequence $x^{(N)}$ converges to $x^{(\infty)}$ if and only if for any fixed $i\in\mathbb{N}$ one has $\lim_{N\to\infty}x_i^{(N)}=x_i^{(\infty)}$. In such a way 
$\mathcal{S}$ is a completely metrizable space and a $\mathcal{S}$-valued random sequence $(w^{(N)})_N$ converges in law to $w^{(\infty)}$ if and only if for any fixed $k$, the truncated sequence $(w_1^{(N)},\cdots,w_k^{(N)},0,\cdots)$ converges in law to $(w_1^{(\infty)},\cdots,w_k^{(\infty)},0,\cdots)$. Therefore  (\ref{EqConvR1.1}) and (\ref{EqConvR1}) imply that 
\begin{equation}\label{EqConvR1.2}
w^{(N)} \tolaw w^{(\infty)}, \quad N\to\infty
\end{equation}
in $\mathcal{S}$. Henceforth we refer to $w^{(N)}$ as the Discrete Disorder of size $N$, and to $w^{(\infty)}$ as the Continuum Disorder.

\subsection{The Energy}
Recalling (\ref{eqIntr3}) we define the rescaled discrete Energy function $\sD:\X\to \mathbb{R}_+$ as
\begin{align}\label{eqResc}
 \sD(\cdot)=\frac{\sigma_{N}(\cdot)}{b_N}=\displaystyle\sum_{i=1}^{N-1} M_i^{(N)}\mathds{1}(Y_i^{(N)}\in \cdot),
 \end{align}
 and (\ref{eqIntr5}) becomes
\begin{equation}\label{eqIntr55bix1}
\Idp{\frac{N^{\gamma}}{b_N}\beta}=\displaystyle\argmax_{I\in\Xdi}\left( \beta \sD(\cdot)-\mathtt{c}E(I)\right),
\end{equation} 
Therefore we choose $\beta_N$ such that
\begin{equation}\label{eqhbxyw}
 \hb_N:=\frac{b_N}{N^{\gamma}}\beta_N
\end{equation}
converges to $\hb\in (0,\infty)$. This is equivalent to relation (\ref{eqIntr7bis}). Since in the sequel we will study the set $\Idp{\frac{N^{\gamma}}{b_N}\beta}$, it is convenient to introduce the notation
\begin{equation}\label{eqIdhat}
\Id{\beta}=\Idp{\frac{N^{\gamma}}{b_N}\beta}.
\end{equation}
In particular $\Id{\hb_N}=\Idp{\beta_N}$.

\begin{remark}
Let us stress that the value of $\mathtt{c}$ is inessential and it can be included in the parameter $\hb$ by a simple rescaling. Therefore from now on we assume $\mathtt{c}=1$.
\end{remark}

It is essential for the sequel to extend the definition of $\Id{\beta}$ to the whole space $\X$ equipped with the Hausdorff metric. This generalization leads us to define the same kind of random set introduced in (\ref{eqIntr55bix1}) in which we use suitable continuum Energy and entropy.

We define the continuum Energy Function $\sC:\X\to \mathbb{R}_+$ as
 \begin{align}\label{EqEnCont}
 \sC(\cdot)=\displaystyle\sum_{i=1}^{\infty} M_i^{(\infty)} \mathds{1}(Y_i^{(\infty)}\in \cdot),
\end{align}
where $(M_i^{(\infty})_{i\in\mathbb{N}}$ and $(Y_i^{(\infty})_{i\in\mathbb{N}}$ are the two independent random sequences introduced in (\ref{EqConvR1.1}) and (\ref{EqConvR1}). 

\begin{remark}\label{rem:aboutMi}
Let us observe that $\sC(I)<\infty$ for all $I\in\X$, because the series $\sum_{i=1}^\infty M_i^{(\infty)}$ converges a.s. Indeed, the law of large numbers ensures that a.s.  $M_i^{(\infty)}\sim i^{-\frac{1}{\alpha}}$ as $i\to\infty$, cf. its definition below \eqref{EqConvR1}, and  $\alpha\in (0,1)$.
\end{remark}

\smallskip

We conclude this section by proving that $\sD$, with $N\in \mathbb{N}\cup\{\infty\}$, is an upper semi-continuous function. 
For this purpose, for $k,N\in \mathbb{N}\cup\{\infty\}$ we define the $k$-truncated Energy function as
\begin{align}\label{eqEntrxyz}
 &\sDtr(\cdot)=\displaystyle\sum_{i=1}^{(N-1)\wedge k} M_i^{(N)} \mathds{1}(Y_i^{(N)}\in \cdot).
\end{align}
Let us stress that the support of $\sDtr$ is given by the space of all possible subsets of $Y^{(N,k)}$, the set of the first $k$-maxima positions
\begin{equation}\label{RemStru11}
Y^{(N,k)}=\{Y_i^{(N)},i=1,2,3,\cdots,(N-1)\wedge k \}\cup \{0,1\}.
\end{equation}
Whenever $k\geq N$ we write simply $Y^{(N)}$. 

\begin{theorem}\label{thmEL}
For any fixed $k, N\in \mathbb{N}\cup\{\infty\}$ and for a.e. realization of the disorder $w^{(N)}$, the function $\sDtr:\X\to \mathbb{R}_+$ is upper semi-continuous (u.s.c.).
\end{theorem}
\begin{remark}\label{remHM1}
 For sake of clarity let us underline that in the Hausdorff metric, cf. \eqref{eqHaus11}, $d_H(A,B)<\varepsilon$ if and only if for any $x_1\in A$ there exists $x_2\in B$ such that $|x_1-x_2|<\varepsilon$ and vice-versa switching the role of $A$ and $B$.
 \end{remark}
 
\begin{proof}
Let us start to consider the case $N\wedge k<\infty$.
For a given $I_0\in \X$, let $\iota$ be the set of all points of $Y^{(N,k)}$ which are not in $I_0$. Since $Y^{(N,k)}$ has a finite number of points there exists $\eta>0$ such that $d(z,I_0)>\eta$ for any $z\in \iota$. Then if $I\in \X$ is sufficiently close to $I_0$, namely $d_H(I,I_0)\leq \eta/2$, then $d(z, I)>\eta/2>0$ for any $z\in\iota$.
Therefore, among the first $k$-maxima, $I$ can at most hit only the points hit by $I_0$, namely $\sDtr(I)\leq \sDtr(I_0)$ and this concludes the proof of this first part.

For the case $N\wedge k=\infty$ it is enough to observe that the difference between the truncated Energy and the original one
 \begin{equation}\label{eql11}
   \displaystyle\sup_{I\in \X}\left|\sC(I)-\sCtr(I)\right|=
   \displaystyle\sup_{I\in \X} \left|\displaystyle\sum_{i=1}^{\infty} M_i^{(\infty)} \mathds{1}(Y_i^{(\infty)}\in I)-\displaystyle\sum_{i=1}^{k} M_i^{(\infty)} \mathds{1}(Y_i^{(\infty)}\in I)\right|\leq \sum_{i>k} M_i^{(\infty)},
  \end{equation}
converges to $0$ as $k\to\infty$ because $\sum_{i} M_i^{(\infty)}$ is a.s. finite, cf. Remark \ref{rem:aboutMi}. Therefore the sequence of u.s.c. functions $\sCtr$ converges uniformly to $\sC$ and this implies the u.s.c. of the limit.
\end{proof}

\subsection{The entropy}\label{THeentropySubsection}
Let us define
\begin{equation}\label{eqDensity1}
\Xden=\{I\in\mathrm{X} : |I|<\infty\}
\end{equation}
and remark that it is a countable dense subset of $\X$ with respect to the Hausdorff Metric.

\medskip

 For a given set $I=\{x_0<x_1<\cdots<x_{\ell}\}\in \Xden$ we define the \emph{entropy} as 
 \begin{equation}\label{defentropyfiniteset}
  E(I)=\sum_{k=1}^{\ell}(x_i-x_{i-1})^{\gamma}.
 \end{equation}
 \begin{theorem}\label{lemmaEntr1}The following holds
 \begin{enumerate}
  \item \label{z1}  The entropy $E(\cdot)$ is strictly increasing with respect to the inclusion of finite sets, namely if $I_1,I_2\in \Xden$ and $I_1\subsetneq I_2$, then $E(I_2)>E(I_1)$,
  \item  \label{z2} The function $E:\Xden \to \mathbb{R}_+$ is lower semi continuous (l.s.c.).
 \end{enumerate}
 \end{theorem} 
 
 \begin{proof}
To prove (\ref{z1}) let us note that if $I_2=\{0, a_1,x,a_2, 1\}$ and $I_2=\{0, a_1,a_2, 1\}$,  with $0\leq a_1<x<a_2\leq 1$ then $E(I_2)-E(I_1)= (x-a_1)^{\gamma}+(a_2-x)^{\gamma}-(a_2-a_1)^{\gamma}>0$ because $\gamma<1$, thus $a^{\gamma}+b^{\gamma}>(a+b)^{\gamma}$ for any $a,b>0$. The claim for the general case follows by a simple induction argument.

To prove (\ref{z2}) we fix $I_0\in \Xden$ and we show that if $(I_n)_n$ is a sequence of finite set converging (in the Hausdorff metric) to $I_0$, then it must be $\liminf_{n\to\infty} E(I_n)\geq E(I_0)$ and by the arbitrariness of the sequence the proof will follow. 

Let $I_0\in\Xden$ be fixed and let us observe that if we fix $\varepsilon>0$ small (precisely smaller than the half of the minimum of the distance between the points of $I_0$), then by Remark \ref{remHM1} any set $I$ for which $d_H(I,I_0)<\varepsilon$ must have at least the same number of points of $I_0$, i.e. $|I|\geq |I_0|$. In such a way if $(I_n)$ is a sequence of finite sets converging to $I_0$, then for any $n$ large enough we can pick out a subset $I_n'$ of $I_n$ with the same number of points of $I_0$ such that $(I_n')_n$ converges to $I_0$. Necessary the points of $I_n'$ converge to the ones of $I_0$, so that $\lim_{n\to\infty}E(I_n')=E(I_0)$.
By using Part (\ref{z1}) we have that for any $n$, $E(I_n)\geq E(I_n')$, so that $\liminf_{n\to\infty}E(I_n)\geq E(I_0)$ and the proof follows.
\end{proof}

We are now ready to define the entropy of a generic set $I\in \X$. The goal is to obtain an extension which conserves the properties of the entropy $E$ on $\Xden$, cf. Theorem \ref{lemmaEntr1}. 
This extension is not trivial because $E$ is strictly l.s.c., namely given $I\in \Xden$ it is always possible to find two sequences $(I_N^{(1)})_N, (I_N^{(2)})_N\in \Xden$ converging to $I$ such that $\lim_{N\to\infty}E(I_N^{(1)})=E(I)$ and $\lim_{N\to\infty}E(I_N^{(2)})=\infty$. 
For instance let us consider the simplest case, when $I=\{0,1\}$. 
Then we may consider $I_N^{(1)}\equiv I$ for any $N$, so that $E(I_N^{(1)})\equiv E(\{0,1\})$, and $I_N^{(2)}$ the set made by $2N$ points such that the first $N$ are equispaced in a neighborhood of $0$ of radius $N^{-\varepsilon}$ and the others $N$ in a neighborhood of $1$ always of radius $N^{-\varepsilon}$, with $\varepsilon=\varepsilon(\gamma)$ small.
 Then $I_N^{(2)}\to I$ as $N\to \infty$ and $E(I_N^{(2)})=2N \cdot 1/N^{\gamma(1+\varepsilon)}+(1-2/N^{\varepsilon})^{\gamma}=O(N^{1-\gamma(1+\varepsilon)}) \to\infty$ as $N\to \infty$ if $\varepsilon < (1-\gamma)/\gamma$.

\smallskip

In order to avoid this problem for $I\in \X$ we define
\begin{equation}\label{eqent2}
 \bar{E}(I)=\liminf_{J\to I, J\in \Xden} E(J).
\end{equation}
Let us stress that $\bar{E}$ is nothing but the smallest l.s.c. extension of $E$ to the whole space $\X$, see e.g. \cite[Prop. 5 TG IV.31]{B71}.
\begin{theorem}\label{ThmentropyCon1}The following hold:
 \begin{enumerate}
  \item \label{zz1} The function $\bar{E}(\cdot)$ is increasing with respect to the inclusion of sets, namely if $I_1,I_2\in \X$ with $I_1\subset I_2$ then $\bar{E}(I_2)\geq \bar{E}(I_1)$. 
  \item \label{zz2} The function $\bar{E}:\X\to \mathbb{R}_+$ is l.s.c. and $\bar{E}\mid_{\Xden}\equiv E$.  
 \end{enumerate}
\end{theorem}

\begin{remark}
 To be more clear
 we recall that 
 \begin{equation}\label{eqent2bis.1}
  \bar{E}(I)= \liminf_{J\to I, J\in \Xden} E(J):=\sup\limits_{\delta> 0}\left[\inf\left\{ E(J) : J\in B_H(I,\delta)\cap \Xden\backslash \{I\}\right\}\right],
 \end{equation}
where $B_H(I,\delta)$ denotes the disc of radius $\delta$ centered on $I$ in the Hausdorff Metric. 

\smallskip

If $ \bar{E}(I)\in \mathbb{R}$ such definition is equivalent to say 
\begin{enumerate}
\item[(a)] \label{H2} For any $\varepsilon>0$ and for any $\delta>0$ there exists $J\in  B_H(\delta,I)\cap \Xden\backslash \{I\}$ such that $\bar{E}(I)+\varepsilon>E(J)$.
 \item [(b)] \label{H1} For any $\varepsilon>0$ there exists $\delta_0>0$ such that for any $J\in  B_H(\delta_0,I)\cap \Xden\backslash \{I\}$, $E(J)>\bar{E}(I)-\varepsilon$.
\end{enumerate}
Note that ($\textrm{a}$) expresses the property to be an  infimum, while ($\textrm{b}$) corresponds to be a  supremum.
\end{remark}
\begin{proof}[Proof of Theorem \ref{ThmentropyCon1}]
We have only to prove (\ref{zz1}). 
Let $I,J\in \X$ such that $J\subset I$. If $\bar{E}(I)=\infty$ there is nothing to prove, therefore we can assume that $\bar{E}(I)\in\mathbb{R}$. 

\smallskip

Let us fix $\varepsilon>0$ and $\delta>0$ (which will be chosen in the sequel). By ($\textrm{a}$) there exists $I'\in \Xden$ such that $\bar{E}(I)+\varepsilon\geq E(I')$ and $d_H(I,I')<\delta$. By the definition of the Hausdorff metric, the family of discs of radius $\delta$ indexed by $I'$ --- $(B(x,\delta))_{x\in I'}$ --- covers $I$ and thus also $J$. Therefore if $J'\subset I'$ is the minimal cover of $J$ obtained from $I'$, i.e. $J':=\min \{L\subset I': J\subset \cup_{x\in L}B(x,\delta)\}$, then it must hold that $d_H(J,J')<\delta$. By Theorem \ref{lemmaEntr1} it follows that $E(I')\geq  E(J')$ and thus $\bar{E}(I)+\varepsilon\geq E(J')$. Let us consider $\bar{E}(J)$ and take $\delta_0>0$ as prescript in ($\textrm{b}$), then as soon as $\delta<\delta_0$, it must hold that $E(J')\geq \bar{E}(J)-\varepsilon$ and this concludes the proof.
\end{proof}

From now on in order to simplify the notation we use $E$ instead of $\bar{E}$ to indicate the function $E$ defined on all $\X$.

 \begin{corollary}\label{CorLSC1}
Let $I\in \X$ such that $E(I)<\infty$. Let $x\notin I$, then $E(I\cup \{x\})>E(I)$. It follows that the function $E $ is strictly increasing whenever it is finite: if $I \subsetneq J$ and $E(I) < \infty$, then $E(I) < E(J)$.
\end{corollary}
\begin{proof} Let $I\in \X$ and let us assume that $E(I)<\infty$. Note that $x\notin I$ means that there exists $\delta>0$ such that $I\cap (x-\delta,x+\delta)=\emptyset$ because $I$ is closed. 
We consider $a,b$ the left and right closest points to $x$ in $I$.
Then the proof will follow by proving that
\begin{equation}\label{eq:corHTEntro1.1}
E(I\cup\{x\})-E(I)\geq (x-a)^{\gamma}+(b-x)^{\gamma}-(b-a)^{\gamma},
\end{equation}
because the r.h.s. is a quantity strictly bigger than $0$, since $\gamma<1$.

\smallskip

To prove \eqref{eq:corHTEntro1.1}, 
we show that the result is true for any finite set in an $\varepsilon$-neighborhood (in the Hausdorff metric) of $I\cup\{x\}$ and then we deduce the result for $E(I)$, by using its definition \eqref{eqent2}. Let us start to observe that for any $\varepsilon$ small enough, if $A$ is a set in an $\varepsilon$-neighborhood of $I\cup\{x\}$, then it can be written as union of two disjoint sets $D,C$ where $D$ is in a $\varepsilon$-neighborhood of $I$ and $C$ in a $\varepsilon$-neighborhood of $\{x\}$. In particular this holds when $A$ is a finite set, and thus 
$$
B_H(I\cup\{x\},\varepsilon)\cap \Xden =\{A\in \Xden\, :\, A=D\cup C,\, D\in B_H(I,\varepsilon) \textrm{ and } C\in B_H(\{x\},\varepsilon)\}.
$$
Furthermore, we can partition any such $D$ in two disjoint sets 
$D'=D\cap [0,x)$ and $D''=(x,1]$.

\smallskip

 For a fixed set $S\in\X$ 
 , let $l_S$ be its smallest point bigger than $0$ and $r_S$ its biggest point smaller than $1$. By using this notation it follows from the definition of the entropy of a finite set (\ref{defentropyfiniteset}) that for any such $A\in B_H(I\cup\{x\},\varepsilon)\cap \Xden$ we have
 
\begin{equation}\label{eq:corHTEntro1.2}
 E(A)=E(D\cup C) =E(D)-(l_{D''}-r_{D'})^{\gamma}+ E(C\cup\{0,1\})-l_C^{\gamma}-(1-r_C)^{\gamma}+(l_C-r_{D'})^{\gamma} +(l_{D''}-r_C)^{\gamma}.
\end{equation}

By Theorem \ref{lemmaEntr1} we can bound $E(C\cup\{0,1\})\geq l_C^{\gamma}+(1-r_C)^{\gamma}+(r_C-l_C)^{\gamma}$. Putting such expression in \eqref{eq:corHTEntro1.1} we obtain

\begin{equation}\label{ze3}
\begin{split}
&E(A)=E(D\cup C)\geq \\
&\qquad E(D)-(l_{D''}-r_{D'})^{\gamma} +(l_C-r_{D'})^{\gamma} +(l_{D''}-r_C)^{\gamma}+ (r_C-l_C)^{\gamma}\geq E(D)+e(\varepsilon),
\end{split}
\end{equation}
where 
$$
e(\varepsilon)=\inf \{(l_C-r_{D'})^{\gamma} +(l_{D''}-r_C)^{\gamma}+ (r_C-l_C)^{\gamma}-(l_{D''}-r_{D'})^{\gamma}\}.
$$ 
Such $\inf$ is taken among all possible $D=D'\cup D'' \in B_H(I,\varepsilon)\cap \Xden$ and $C\in B_H(\{x\},\varepsilon)\})\cap \Xden$. 

Finally (\ref{ze3}) implies that  $\inf E(A)\geq \inf E(D) +e(\varepsilon)$, where the $\inf$ is taken among all possible $A=D\cup C\in B_H(I\cup\{x\},\varepsilon)\cap \Xden\backslash \{I\cup\{x\}\}$. By taking the limit for $\varepsilon\to 0$ we have $e(\varepsilon)\to (x-a)^{\gamma}+(b-x)^{\gamma}-(b-a)^{\gamma}$ and the result follows by \eqref{eqent2bis.1},  since the r.h.s. of (\ref{ze3}) is independent of $C$.
\end{proof}
\begin{proposition}
 For any $0\leq a < b\leq 1$ we have that $E([a,b])=\infty$.
\end{proposition}
\begin{proof}Let us consider the case in which $a=0,b=1$, the other cases follow in a similar way.
 By Theorem \ref{lemmaEntr1} we have that $E([0,1])\geq E(\{0,1/N,\cdots,1\})=N^{1-\gamma}\uparrow\infty$ as $N\uparrow\infty$ because $\gamma<1$.
\end{proof}

\subsection{The Energy-entropy}\label{sec:EEdef1}
\begin{definition}\label{EEdef1}
For any $N,k\in\mathbb{N}\cup\{\infty\}$ and $\beta\in (0,\infty)$ we define, cf. \eqref{eqEntrxyz} and \eqref{eqent2},
\begin{equation}\label{eqVarEExyzbix}
 \Udtr{\beta}(I)=\beta \sDtr(I)-E(I).
\end{equation}
Note that $\Udtr{\beta}$ is upper semi-continuous on $(\X,d_H)$, a compact metric space, therefore its maximizer
\begin{equation}\label{eqVarEExyz}
 \udtr{\beta}=\max_{I\in\X} \Udtr{\beta}(I).
\end{equation}
is well defined.
\end{definition}
Whenever $k\geq N$ we will omit the  superscript $(k)$ from the notation.

\begin{theorem}\label{ThmExistence1}
For any $N,k\in\mathbb{N}\cup\{\infty\}$, $\beta>0$ and for a.e. realization of the disorder $w^{(N)}$, the maximum $\udtr{\beta}$ is achieved in only one set, i.e. the solution at
   \begin{equation}\label{eqVarEE}
 \Idtr{\beta}=\argmax_{I\in\X} \Udtr{\beta}(I)
\end{equation}
is unique.
Moreover for any $N\in\mathbb{N}$ we have that $\Idtr{\beta}\in \Xdi$.
\end{theorem}
\begin{proof}
We claim that if $I$ is a solution of (\ref{eqVarEE}), then by using Corollary \ref{CorLSC1} 
\begin{align}
& I\subset  Y^{(N,k)}, \quad \textrm{if } N\wedge k <\infty,\label{eqUni1}\\
& I = \overline{I\cap Y^{(\infty)}}\quad \textrm{if } N\wedge k =\infty.\label{eqUni1infty}
\end{align}
 Indeed if $N\wedge k<\infty$ and (\ref{eqUni1}) fails, then there exists $x\in I$ such that $x\notin Y^{(N,k)}$ and this implies $\sDtr(I)=\sDtr(I-\{x\})$, but $E(I-\{x\})<E(I)$ by Corollary \ref{CorLSC1}. Therefore $\Udtr{\beta}(I-\{x\})>\Udtr{\beta}(I)=\udtr{\beta}$, a contradiction.
 The case $N\wedge k=\infty$ follows in an analogous way always by using Corollary \ref{CorLSC1}, because the set in the r.h.s. of (\ref{eqUni1infty}), which is a subset of $I$, has the same Energy as $I$ but smaller entropy.
 Now we are able to conclude the uniqueness, by following the same ideas used in \cite[Proposition 4.1]{HM07} or \cite[Lemma 4.1]{AL11}:
let $I^1$, $I^2$ be two subsets achieving the maximum. By using (\ref{eqUni1}) and (\ref{eqUni1infty}) if $I^1\neq I^2$, then there would exist $Y_j^{(N)}$ such that $Y_j^{(N)}\in {I_1}$ and $Y_j^{(N)}\notin {I_2}$. Note that if $N\wedge k =\infty$, by \eqref{eqUni1infty} we can assume $Y_j^{(N)}\in Y^{(\infty)}$, so that
 \begin{equation}\label{eqUniqueness}
  \displaystyle\max_{I:Y_j^{(N)}\in I} \Udtr{\beta}(I)=\displaystyle\max_{I:Y_j^{(N)}\notin I} \Udtr{\beta}(I)
 \end{equation}
 and this leads to
 \begin{align}\label{eqUniqueness2}
  \beta M_j^{(N)}=\beta \sDtr(Y_j^{(N)})=\displaystyle\max_{I:Y_j^{(N)}\notin I} \Udtr{\beta}(I)-
  \displaystyle\max_{I:Y_j^{(N)}\in I} \left\{\beta \displaystyle\sum_{k\neq j:Y_k^{(N)}\in I} M_k^{(N)}-E(I)\right\}.
 \end{align}
 Let us stress that the r.h.s. is independent of $M_j^{(N)}$, which is on the l.h.s. Then, by conditioning on the values of $(M_i^{(N)})_{i\in\mathbb{N},i\neq j}$ and $(Y_i^{(N)})_{i\in\mathbb{N}}$ we have that the l.h.s. has a continuous distribution, while the r.h.s. is a constant, so that the event in which the r.h.s. is equal to the l.h.s. has zero probability.
 By countable sub-additivity of the probability we have that a.s. ${I_1} = {I_2}$.
\end{proof}

\section{Convergence}\label{SecConver}
The aim of this section is to discuss the convergence of $\Id{\hb_N}$, \eqref{eqVarEE}, and $\ud{\hb_N}$, \eqref{eqVarEExyz}, when $\lim_{N\to\infty}\hb_N=\hb\in (0,\infty)$, cf. \eqref{eqhbxyw}.

\smallskip

For technical convenience we build a coupling between the discrete disorder and the continuum one. We recall that by (\ref{EqConvR1.2}) $w^{(N)}$ converges in distribution to $w^{(\infty)}$ on $\mathcal{S}$, a completely metrizable space. Therefore by using Skorokhod's representation Theorem (see \cite[theorem 6.7]{B99}) we can define $w^{(N)}$ and $w^{(\infty)}$ on a common probability space in order to assume that their convergence holds almost surely.
\begin{lemma}\label{RSko}
There is a coupling (that, with a slight abuse of notation, we still call $\mathbb{P}$) of the continuum model and the discrete one, under which 
\begin{equation}\label{eqCoupling}
 w^{(N)}=(M_i^{(N)}, Y_i^{(N)})_{i\in \mathbb{N}}\xrightarrow[\mathbb{P}-\textrm{a.s.}]{\mathcal{S}} w^{(\infty)}=(M_i^{(\infty)}, Y_i^{(\infty)})_{i\in\mathbb{N}}, \textrm{ as } N\to \infty.
\end{equation}
In particular for any fixed $\varepsilon, \delta>0$ and $k\in\mathbb{N}$ there exists $\hat{N}<\infty$ such that for all $N>\hat{N}$
\begin{align}\label{eqRSko}
 &\mathbb{P}\left(\displaystyle\sum_{j=1}^{(N-1)\wedge k}\mid M_j^{(N)}-M_j^{(\infty)}\mid <\varepsilon\right)>1-\delta, \\
  &\mathbb{P}\left(\displaystyle\sum_{j=1}^{(N-1)\wedge k}\mid Y_j^{(N)}-Y_j^{(\infty)}\mid <\varepsilon\right)>1-\delta, 
\end{align}
\end{lemma}

\subsection{Convergence Results}
Let us rewrite an equivalent, but more handy definition of $\Idtr{\beta}$ and $\udtr{\beta}$:
for a given $k\in \mathbb{N}$ let
\begin{align}
\mathcal{C}_{k}=\{A : A\subset\{1,\cdots, k\} \}
\end{align}
and for any $k\in \mathbb{N}$, $N\in\mathbb{N}\cup\{\infty\}$ and $A\subset\{1,\cdots,k\}$ let $Y_A^{(N)}={\{Y_i^{(N)}\}_{i\in A}\cup\{0,1\}}$, which is well defined also for $A=\emptyset$. Therefore by Theorem \ref{ThmExistence1} we can write
\begin{equation}\label{eqHT:ANK}
	\begin{split}
 &\udtr{\beta}=\max_{A\in \mathcal{C}_{k}}\left[\beta \sum_{i\in A}M_i^{(N)}-E(Y_A)\right],\\
 &\Idtr{\beta}=Y_{A_{\beta,N}^{(k)}}^{(N)},
\end{split}
\end{equation}
for a suitable random set of indexes $A_{\beta,N}^{(k)}$ (which can be empty or not). We have our first convergence result.
\begin{proposition}\label{PropConA} 
Assume that $\hb_N\to \hb$ as $N\to\infty$. Then for any fixed $\delta>0$ and $k\in \mathbb{N}$ there exists $N_k$ such that for any $N>N_k$
\begin{equation}
 \mathbb{P}\left(A_{\hb_N,N}^{(k)}= A_{\hb,\infty}^{(k)}\right)>1-\delta.
\end{equation}
\end{proposition}
\begin{proof}
To prove the claim by using the sub-additivity of the probability, it is enough to prove that for any $r\in\{1,\cdots,k\}$ 
\begin{align}
 &\mathbb{P}\left(r\not\in A_{\hb_N,N}^{(k)}, r\in A_{\hb,\infty}^{(k)}\right)\to 0,\quad N\to\infty,\\
 &\mathbb{P}\left(r\in A_{\hb_N,N}^{(k)}, r\not\in A_{\hb,\infty}^{(k)}\right)\to 0,\quad N\to\infty.
\end{align}

We detail the first one, the second one follows in an analogous way. On the event $\{r\not\in A_{\hb_N,N}^{(k)}, r\in A_{\hb,\infty}^{(k)}\}$ we consider 
\begin{equation}\label{spe1.1}
 \hat{u}_{(r)}:=\max_{A\in\mathcal{C}_k, r\not\in A}\left[\hb\sum_{i\in A} M_i^{(\infty)}-E(Y_A)\right]<\uctr{\hb}                                                                                 
\end{equation}
because $r\in A_{\hb,\infty}^{(k)}$ and  the set that achieves the maximum is unique.
Then
\begin{multline}\label{spe1}
 \max_{A\in\mathcal{C}_k, r\not\in A}\left[\hb_N\sum_{i\in A} M_i^{(N)}-E(Y_A)\right]\\
  \leq
 \max_{A\in\mathcal{C}_k, r\not\in A}\left[\hb\sum_{i\in A} M_i^{(\infty)}-E(Y_A)\right]+
 |\hb_N-\hb| \sum_{i=1}^k M_i^{(N)}+\hb \sum_{i=1}^k |M_i^{(N)}-M_i^{(\infty)}|\\
 =\hat{u}_{(r)}+|\hb_N-\hb| \sum_{i=1}^k M_i^{(N)}+\hb \sum_{i=1}^k |M_i^{(N)}-M_i^{(\infty)}|
\end{multline}
and in the same way, always on the event $\{r\not\in A_{\hb_N,N}^{(k)}, r\in A_{\hb,\infty}^{(k)}\}$,
\begin{equation}\label{spe2}
\max_{A\in\mathcal{C}_k, r\in A}\left[\hb_N\sum_{i\in A} M_i^{(N)}-E(Y_A)\right]\geq \uctr{\hb}-|\hb_N-\hb| \sum_{i=1}^k M_i^{(N)}-\hb \sum_{i=1}^k |M_i^{(N)}-M_i^{(\infty)}|.
\end{equation}
Therefore by using the assumption that $r\not\in A_{\hb_N,N}^{(k)}$ we have that the l.h.s. of (\ref{spe1}) is larger than the l.h.s. of (\ref{spe2}). Together with \eqref{spe1.1} we obtain $0<\uctr{\hb}-\hat{u}_{(r)}\leq 2\hb \sum_{i=1}^k |M_i^{(N)}-M_i^{(\infty)}|+2|\hb_N-\hb| \sum_{i=1}^k M_i^{(N)}$, and a simple inclusion of events gives
\begin{equation}
 \mathbb{P}\left(r\not\in A_{\hb_N,N}^{(k)}, r\in A_{\hb,\infty}^{(k)}\right)\leq \mathbb{P}\left(0<\uctr{\hb}-\hat{u}_{(r)} \leq 2\hb \sum_{i=1}^k |M_i^{(N)}-M_i^{(\infty)}|+2|\hb_N-\hb| \sum_{i=1}^k M_i^{(N)}\right).
\end{equation}
The proof follows by observing that the r.h.s. converges to $0$ as $N\to\infty$ by Lemma \ref{RSko}.
\end{proof}

The following proposition contains the convergence results for the truncated quantities $\Idtr{\hb_N}$ and $\udtr{\hb_N}$, cf. \eqref{eqVarEE} and \eqref{eqVarEExyz} respectively.

\smallskip

We introduce the maximum of $\Udtr{\beta}$, cf. \eqref{eqVarEExyzbix}, outside a neighborhood of radius $\delta$ of $\Idtr{\beta}$
\begin{definition}\label{defudtrdek}
 For any $\delta>0$, ${\beta}\in (0,\infty)$ we define
  \begin{equation}
   \udtr{\beta}(\delta)=\max_{I\in \X: d_H(I,\Idtr{\beta})\geq \delta} \Udtr{\beta}(I),
  \end{equation}
  where $\Udtr{\beta}$ is defined in \eqref{eqVarEExyzbix}.
\end{definition}
\begin{proposition}\label{PropConU}Assume that $\hb_N\to \hb$ as $N\to\infty$. The following hold
 \begin{enumerate} 
   \item \label{con5} For every fixed $\delta>0$, ${\beta}\in (0,\infty)$ 
$\mathbb{P}\left(\liminf\limits_{k\to\infty} (\uctr{\beta}-\uctr{\beta}(\delta))>0\right)=1$.
  \item \label{con4} For any $\varepsilon, \delta>0$ and for any fixed $k$ there exists $N_k$ such that 
   $\mathbb{P}\left(|\udtr{\hb_N}-\uctr{\hb}|<\varepsilon\right)>1-\delta,$
  for any $N>N_k$.
  \item \label{con1} For any $\varepsilon, \delta>0$ and for any fixed $k$ there exists $N_k$ such that 
   $\mathbb{P}\left(d_H(\Idtr{\hb_N},\Ictr{\hb})<\varepsilon\right)>1-\delta,$
  for any $N>N_k$.
    \item \label{con6} For any $\varepsilon, \delta>0$, there exist $\eta, K>0$ and $(N_k)_{k>K}$, such that 
$\mathbb{P}\left(\udtr{\hb_N}(\varepsilon)<\udtr{\hb_N}-\eta \right)>1-\delta,$
  for any $k>K$, and $N>N_k$.
 \end{enumerate}
\end{proposition}
\begin{proof}
  We follow \cite[Part (3,4) of Proof of Lemma 4.1]{AL11}.
By contradiction if there exists $\delta>0$ such that $\liminf_{k\to\infty} (\uctr{\beta}-\uctr{\beta}(\delta))=0$, then we may find a sequence $I_{k_j}$ such that $\limsup_{j\to\infty} U_{\beta,\infty}^{(k_j)}(I_{k_j})\geq \liminf_{j\to\infty} U_{\beta, \infty}^{ (k_j)}(\hat{I}_{\beta,\infty}^{(k_j)})$ and $d_H(\hat{I}_{\beta, \infty}^{k_j},I_{k_j})>\delta$.
By compactness of the space $\X$ we can suppose that there exists $I_0\in \X$ such that $\lim_{j\to\infty}I_{k_j}= I_0$, therefore by using the u.s.c. property of $\Udtr{\beta}$, cf. Section \ref{sec:EEdef1}, that for any fixed $k\in \mathbb{N}$, $\Uc{\beta}(I)\geq \Uctr{\beta}(I)$ and $\Uctr{\beta}(I)\uparrow \Uc{\beta}(I)$ as $k\uparrow\infty$, we get
\begin{multline}
\Uc{\beta}(I_0)\geq \limsup_{j\to \infty} \Uc{\beta}(I_{k_j})\geq \limsup_{j\to\infty} U_{\beta,\infty}^{(k_j)}(I_{k_j}) \geq \\ \geq \liminf_{j\to\infty} U_{\beta, \infty}^{ (k_j)}(\hat{I}_{\beta,\infty}^{(k_j)}) 
\geq  \liminf_{j\to\infty} U_{\beta, \infty}^{ (k_j)}(\Ic{\beta})
=\Uc{\beta}(\Ic{\beta})=\uc{\beta},
\end{multline}
namely, $\Uc{\beta}(I_0)=\uc{\beta}$. The uniqueness of the maximizer, cf. Theorem \ref{ThmExistence1}, implies $I_0=\Ic{\beta}$.
Thus if we show that $\lim_{k\to\infty}\Ictr{\beta}=\Ic{\beta}$, then we obtain the desired contradiction, because the two sequences $(I_{k_j})_j$ and $(\hat{I}_{\beta,\infty}^{(k_j)})_j$ are at distance at least $\delta$ therefore they cannot converge to the same limit. By compactness of $\X$ we can assume that $\Ictr{\beta}$ converges to $I_1$.
Therefore, again by u.s.c. of $\Uc{\beta}$, we get
\begin{equation}
\label{eqCo1}
 \Uc{\beta}(I_1)\geq \limsup_{k\to \infty} \Uc{\beta}(\Ictr{\beta})\geq  
 \limsup_{k\to \infty} \Uctr{\beta}(\Ictr{\beta})
\geq \limsup_{k\to \infty} \Uctr{\beta}(\Ic{\beta})=\Uc{\beta}(\Ic{\beta}).
\end{equation}
The uniqueness of the maximizer forces $\Ic{\beta}=I_1$ and this concludes the proof.$\hfill\Box$

\smallskip

To prove Part (\ref{con4}) we observe that 
\begin{align}
 &\uctr{\hb}=\max_{A\in\mathcal{C}_k}\left[\hb \sum_{i\in A}M^{(\infty)}_i-E(Y_A)\right]\leq \udtr{\hb_N}+|\hb_N-\hb|\sum_{i=1}^k M_i+\hb_N\sum_{i=1}^k|M_i^{(\infty)}-M_i^{(N)}|,\\
 &\udtr{\hb_N}=\max_{A\in\mathcal{C}_k}\left[\hb_N \sum_{i\in A}M_i^{(N)}-E(Y_A)\right]\leq \uctr{\hb}+|\hb_N-\hb|\sum_{i=1}^k M_i^{(N)}+\hb\sum_{i=1}^k|M_i^{(\infty)}-M_i^{(N)}|.
\end{align}
and the proof follows by Lemma \ref{RSko} and the assumption on $\hb_N$.$\hfill\Box$

\smallskip

To prove Part (\ref{con1}) we observe that by Lemma \ref{RSko} for any fixed $\varepsilon, \delta>0$ and $k\in\mathbb{N}$, there exists $N_k$ such that, for all $N>N_k$, 
$\mathbb{P}\left(d(Y_i^{(k)},Y_i^{(\infty)})<\varepsilon,\, \text{for any }\, i=1,\cdots, k\right)>1-\delta/2 .$ By Proposition \ref{PropConA} we can furthermore suppose that for any $N>N_k$, $\mathbb{P}\left(A_{\hb_N,N}^{(k)}=A_{\hb,\infty}^{(k)}\right)>1-\delta/2$, cf. \eqref{eqHT:ANK}. The intersection of such events gives the result.
  $\hfill\Box$
  
  \smallskip
  
 To prove Part \eqref{con6} 
 we prove first an intermediate result: for any given $\delta, \varepsilon,\eta >0$, and $k\in\mathbb{N}$ there exists $N_k$ such that 
 \begin{equation}\label{eq9.infui}
\udtr{\hb_N}(\varepsilon)<\uctr{\hb}(\varepsilon/4)+\eta/4,
\end{equation} 
with probability larger than $1-\delta/2$, for all $N>N_k$.

For this purpose, by 
Part (\ref{con1}), for any $k>0$ there exists $N_k>0$ such that for all $N>N_k$, $d_H(\Idtr{\hb_N},\Ictr{\hb})<\frac{\varepsilon}{4}$ with probability larger than $1-\delta/4$. 
Let $I$ be a set achieving $\udtr{\hb_N}(\varepsilon)$, so that by definition $d_H(I,\Idtr{\hb_N})\geq\varepsilon$. 
It is not difficult to see that $I\subset Y^{(N,k)}$ (points outside $Y^{(N,k)}$ does not contribute to the Energy, but increase the entropy). 
We claim that for any $\eta>0$ there exists $I'\subset Y^{(\infty,k)}\in \Xden$ such that $d_H(I',I)<{\varepsilon}/2$ and $\Udtr{\hb_N}(I)\leq \Uctr{\hb}(I')+\eta/4$ 
with probability larger than $1-\delta/4$ . 
 This relation implies that $\udtr{\hb_N}(\varepsilon)\leq \uctr{\hb}(\varepsilon/4)+\eta/4$, because $d_H(I',\Ictr{\hb})>\varepsilon/4$ and \eqref{eq9.infui} follows.
 The existence of $I'$ is explicit: we observe that $I=\{0,Y_{i_1}^{(N)},\cdots ,Y_{i_{\ell}}^{(N)},1\}$, for a suitable choice of indexes $\{i_1,\cdots,i_{\ell}\}\subset\{1,\cdots,k\}$.  By using Lemma \ref{RSko} it is not difficult to show that we can choose $I'=\{0,Y_{i_1}^{(\infty)},\cdots ,Y_{i_{\ell}}^{(\infty)},1\}$, possibly by enlarging $N$. 
 
 The proof of \eqref{con6} follows by observing that by 
 Part (\ref{con5}), there exist $\eta>0$ and $K>0$ such that $\uctr{\hb}(\varepsilon/4)\leq \uctr{\hb}-\eta$ with probability larger than $1-\delta/4$, for any $k>K$. This provides an upper bound for (\ref{eq9.infui}) and Part (\ref{con4}) allows to complete the proof.
\end{proof}

Let us stress that for any fixed $N\in \mathbb{N}$ we have that $\Idtr{\hb_N}\equiv\Id{\hb_N}$ as $k>N$. In the following Proposition we show that this convergence holds uniformly on $N$.

\begin{proposition}\label{PropConI}The following holds
 \begin{enumerate}
 \item \label{Tcnlemma3} For any $N,k\in \mathbb{N}\cup\{\infty\}$ we define
 \begin{equation}\label{eqRho1}
 \Rnk:=\displaystyle\sup_{I\in \X} \left|\sD(I)-\sDtr(I)\right|=\sum_{i>k} M_i^{(N)}.
 \end{equation}
Then for any $\varepsilon,\delta>0$ there exists $K>0$ such that
$\mathbb{P}(\Rnk>\varepsilon)<\delta$ for all $k>K$, uniformly on $N\in\mathbb{N}$.
\item \label{con5.1}$\mathbb{P}\left(\displaystyle\lim_{k\to\infty}\Ictr{\hb}=\Ic{\hb}\right)=1.$
 \item \label{con2} For any $\varepsilon, \delta>0$ there exists $K>0$ such that 
$\mathbb{P}\left(d_H(\Idtr{\hb_N},\Id{\hb_N})<\varepsilon\right)>1-\delta$ for all $k>K$, uniformly on $N$.
 \end{enumerate}
\end{proposition}
\begin{proof}
Part (\ref{Tcnlemma3}) is similar to \cite[Proposition 3.3]{HM07} and actually simpler. Here we give a short sketch of the proof. We note that if $k\geq N$, then $\Rnk \equiv 0$, therefore we can suppose $k<N$. For such $k$ we consider the "good event", like in \cite[(3.8)]{HM07}
\begin{equation}
\mathcal{B}_k^{(N)}=\left\{F^{-1}\left(1-\frac{2r}{N}\right)\leq \tilde{M}_r^{(N)}\leq F^{-1}\left(1-\frac{1}{N}\right), \, \text{ for all $k\leq r\leq N-1$} \right\}.
\end{equation}
Then, cf. \cite[Lemma 3.4]{HM07} $\mathbb{P}\left(\mathcal{B}_k^{(N)}\right)\to 1$ as $k\to\infty$, uniformly on $N$. By partitioning with respect to the "good event" and then by using Markov's inequality, we get that for any $\varepsilon>0$
\begin{equation}\label{eq:lemma3.3remake}
\mathbb{P}\left(\Rnk >\varepsilon\right)\, \leq \, \mathbb{P}\left(\mathcal{B}_k^{(N)}\, \text{fails}\right)+ \varepsilon^{-1}\, \sum_{r=k}^{N-1}\mathbb{E}\left[M_r^{(N)}; \mathcal{B}_k^{(N)}\right].
\end{equation}
To conclude the proof it is enough to show that $\sum_{r=k}^{N-1}\mathbb{E}\left[M_r^{(N)}; \mathcal{B}_k^{(N)}\right]$ 
converges to $0$ as $k\to\infty$, uniformly on $N>k$.
An upper bound for $\mathbb{E}\left[M_r^{(N)}; \mathcal{B}_k^{(N)}\right]$ is provided by \cite[Lemma 3.8]{HM07}: for any $\delta>0$ there exist $c_0,c_1$ and $c_2>0$ such that for any $2(1+1/\alpha)<k<r<N$
$$
\mathbb{E}\left[M_r^{(N)}; \mathcal{B}_k^{(N)}\right]\, \leq\, c_0 r^{-\frac{1}{\alpha}+\delta}+c_1 b_N^{-1}\ind_{\{r>c_2n\}}.
$$
This allows to conclude that there exist $c_0',c_1'>0$ such that
$$
\sum_{r=k}^{N-1}\mathbb{E}\left[M_r^{(N)}; \mathcal{B}_k^{(N)}\right]\leq c_0' k^{-\frac{1}{\alpha}+1+\delta}+ c_1' Nb_N^{-1}.
$$
Since $\alpha\in (0,1)$ and $ Nb_N^{-1}\to 0$ as $N\to\infty$, cf. \eqref{eq:b_Nbev}, we conclude that the r.h.s. converges to $0$ as $k\to\infty$, uniformly on $N>k$.
$\hfill\Box$

\smallskip

Part (\ref{con5.1}) has been already proven in the proof of Part (\ref{con5}) of Proposition \ref{PropConU}.  $\hfill\Box$

\smallskip


 
 Part \eqref{con2} is a consequence of \eqref{con6}.
Let us fix $k$ such that \eqref{con6} holds for any $N>N_k$ and that $\mathbb{P}(\hb_N\Rnk<\eta/4)>1-\delta$ uniformly on $N$, cf. \eqref{eqRho1}. In such case we claim that, for any $\ell>k$ and $N>N_k$
 \begin{equation}\label{prof:ciccio3}
 d_H(\Idtr{\hb_N},\hat{I}_{{\hb_N}, N}^{(\ell)})<\varepsilon,
 \end{equation}
 with probability larger than $1-2\delta$.
Otherwise if $d_H(\Idtr{\hb_N},\hat{I}_{{\hb_N}, N}^{(\ell)})\geq \varepsilon$ for some $\ell>k$, then it holds that 
\begin{equation}\label{prof:cicio2}
\hat{u}_{\hb_N,N}^{(\ell)}\,\leq\, U_{\hb_N,N}^{(k)}(\hat{I}_{{\hb_N}, N}^{(\ell)})+\hb_N\Rnk\,\leq\, \udtr{\hb_N}(\varepsilon)+\eta/4.
\end{equation}
Relation \eqref{con6} provides an upper bound for the r.h.s. of \eqref{prof:cicio2}, giving $\hat{u}_{\hb_N,N}^{(\ell)}\leq \udtr{\hb_N}-\eta/4$ and this is a contradiction because $\ell\mapsto \hat{u}_{\hb_N,N}^{(\ell)}$ is non-decreasing and thus $\hat{u}_{\hb_N,N}^{(\ell)}\geq \udtr{\hb_N}$. By using \eqref{prof:ciccio3} together with the triangle inequality we conclude that 
for any $\ell>k$ and $N>N_k$
 \begin{equation}\label{prof:ciccio4}
 d_H(\Id{\hb_N},\hat{I}_{{\hb_N}, N}^{(\ell)})<2\varepsilon,
 \end{equation}
 with probability larger than $1-4\delta$. To conclude we have to consider the case in which $N\,\leq\,N_k$. For any such $N$, $\Idtr{\hb_N}$ converges to $\Id{\hb_N}$ as $k\to\infty$. To be more precise, whenever $k>N$ we have that $\Idtr{\hb}=\Id{\hb}$. This concludes the proof. 	
%
 \end{proof}

\subsection{Proof of theorem \ref{TheoremMain1}}
 The proof is a consequence of \cite[Theorem 3.2]{B99}, which can be written as follows
 \begin{theorem}
  Let us suppose that the r.v's $X_N^{(k)}, X^{(k)}, X_N, X$  take values in a separable metric space $(\mathcal{S}, d_{\mathcal{S}})$ and $X_N^{(k)}, X_N$ are defined on the same probability space. Then if the following diagram holds
  \begin{displaymath}
    \xymatrix{ X_N^{(k)} \ar[d]^{k\to\infty}_{\textrm{in probability, uniformly in } N} \ar[r]_{N\to\infty}^{(\mathrm{d})} & X^{(k)} \ar[d]_{(\mathrm{d})}^{k\to\infty} \\
               X_N  & X  }
\end{displaymath}
then $X_N{\tolaw}X$.
The expression \emph{in probability, uniformly in $N$} means 
\begin{equation}
 \lim_{k\to\infty}\limsup_{N\to\infty}\mathbb{P}\left(d_{\mathcal{S}}(X_N^{(k)}, X_N)\geq \varepsilon \right)=0,
\end{equation}
for any fixed $\varepsilon>0$.
 \end{theorem}
In our case we have $X_N^{(k)}=\Idtr{\hb_N}$, $X^{(k)}=\Ictr{\hb}$, $X_N=\Id{\hb_N}$ and $X=\Ic{\hb}$ and by Propositions \ref {PropConU} and \ref{PropConI} the diagram above holds.
\begin{remark}\label{RemOnTh1}
Let us stress that under the coupling introduced in Lemma \ref{RSko} we have that in Theorem \ref{TheoremMain1} the convergence of $\Id{\hb_N}$ to $\Ic{\hb}$ holds in probability, namely for any $\varepsilon, \delta>0$ one has $\mathbb{P}\left(d_H(\Id{\hb_N},\Ic{\hb})<\varepsilon\right)>1-\delta$ for all $N$ large enough. This follows by Part (\ref{con1}) of Proposition \ref{PropConU} and Parts (\ref{con5.1}), (\ref{con2}) of Proposition \ref{PropConI}.
\end{remark}

\section{Concentration}\label{SecConcen}
In this section we discuss the concentration of $\tau/N\cap [0,1]$ around the set $\Id{\hb_N}$, cf. (\ref{eqIdhat}), giving a proof of Theorems \ref{TheoremMain2}, \ref{TheoremMain3}.

\subsection{General Setting of the Section}\label{Sech3} Let us stress that in the pinning model (\ref{eqIntr1}) we can replace $h>0$ in the exponent of the Radon-Nikodym derivative by $h=0$ by replacing the original renewal $\tau$ with a new one, $\tilde{\tau}$ defined by $\mathrm{P}(\tilde{\tau}_1=n)=e^{h}\mathrm{P}(\tau_1=n)$ and $\mathrm{P}(\tilde{\tau}_1=\infty)=1-e^{h}$. Note that the renewal process $\tilde{\tau}$ is terminating because $h<0$.
In this case (cf. 
\ref{Aa1}) the renewal function $\tilde{u}(n):=\mathrm{P}(n\in\tilde{\tau})$ satisfies
\begin{equation}\label{eqIntr8}
 \lim_{n\to\infty} \frac{\log \tilde{u}(n)}{ n^{\gamma}}=-\mathtt{c},
\end{equation}
with the same $\gamma$ and $\mathtt{c}$ used in Assumptions \ref{ASSrenP} for the original renewal process $\tau$. 

In the sequel we assume $\mathtt{c}=1$ (as already discussed in Section \ref{SecConver}), $h=0$ and we omit the tilde-sign on the notations, writing simply $\tau$ and $u(\cdot)$ instead of $\tilde{\tau}$ and $\tilde{u}(\cdot)$.

\subsection{Proof of Theorem \ref{TheoremMain2}}
To prove Theorem \ref{TheoremMain2} we proceed in two steps. In the first one we consider a truncated version of the Gibbs measure (\ref{eqIntr1}) in which we regard only the first $k$-maxima among $\omega_1,\cdots,\omega_{N-1}$ and we prove concentration for such truncated pinning model, cf. Lemma \ref{lemmaConc}. In the second step we show how to deduce Theorem \ref{TheoremMain2}.

\smallskip

Let us define the truncated pinning model. For technical reasons it is useful to write the energy using $\hat{\sigma}_N$ defined in (\ref{eqResc}).

\begin{definition}\label{DefTruGibb} For $N,k \in \mathbb{N}$, $\beta>0$, the $k$-truncated Pinning Model measure is a probability measure defined by the following Radon-Nikodym derivative
\begin{equation}\label{eqGibbrestr}
\frac{\textrm{d} \hat{\mathrm{P}}_{\beta,N}^{(k)} }{\textrm{d}\mathrm{P}_N}(I)=
 \frac{e^{N^{\gamma}\beta\sDtr(I)}\mathds{1}(1\in I)}{\hat{\mathrm{Z}}_{\beta,N}^{(k)}},
\end{equation}
where $\mathrm{P}_N$ is the law of $\tau/N\cap[0,1]$ used in (\ref{eqIntr1}).

In the sequel we use the convention that whenever $k\geq N$, the superscript $(k)$ will be omitted.
\end{definition}
\begin{remark}
Note that whenever $\beta=\hat{\beta}_N$, the Radon-Nikodym derivative \eqref{eqGibbrestr} with $k\geq N$ recovers the original definition \eqref{eqIntr1} with $\beta=\beta_N$.
\end{remark}

\begin{lemma}\label{lemmaConc}
 Let $(\hat{\beta}_N)_N$ be a sequence converging to $\hat{\beta}\in (0,\infty)$.
 For any fixed $\varepsilon, \delta>0$ there exist $\nu=\nu(\varepsilon,\delta)>0$, $K=K(\varepsilon,\delta)$ and  $({N}_k)_{k\geq K}$ such that 
 \begin{equation}
   \mathbb{P}\left( \hat{\mathrm{P}}_{\hb_N,N}^{(k)}\left(d_H(I,\Idtr{\hb_N})>\delta\right)\leq e^{-N^{\gamma}\nu}\right)>1-\varepsilon
 \end{equation}
for all $k>K$ and $N>{N}_k$.
\end{lemma}
 Roughly speaking to prove Lemma \ref{lemmaConc} we need to estimate the probability that a given set $\iota=\{\iota_1,\cdots,\iota_{\ell}\}$, with $\iota_j<\iota_{j+1}$, is contained in $\tau/N$. In other words we need to compute the probability that $\iota_1,\cdots,\iota_{\ell}\in \tau/N$ when $N$ is large enough. 

For this purpose we fix $\iota=\{\iota_1,\cdots,\iota_{\ell}\}\subset [0,1]$ and we consider $\iota^{(N)}=\{\iota_1^{(N)},\cdots,\iota_{\ell}^{(N)}\}$, where $\iota_i^{(N)}$ is the nearest point to $\iota_i$ in the lattice $\{0,1/N,\cdots, 1\}$. We define $u_N(\iota)=\prod_{i=1}^{\ell}u(N(\iota_i^{(N)}-\iota_{i-1}^{(N)}))$ 
, with $\iota_0^{(N)}:=0$.
The behavior of $u_N(\iota)$ as $N\to\infty$ is given by the following result
\begin{proposition}\label{propoConc1}
Let $\iota=\{\iota_1,\cdots,\iota_{\ell}\}\subset [0,1]$ be a fixed and finite set and consider the associated real sequence $(u_N	(\iota))_N$. Then $\lim_{N\to\infty}\frac{1}{N^{\gamma}}\log u_N(\iota)=-\sum_{i=1}^{\ell}(\iota_{i}-\iota_{i-1})^{\gamma}$ and it holds uniformly in the space of all subsets $\iota$ with points spaced at least by $\xi$, for any fixed $\xi>0$.
\end{proposition} 
\begin{proof}
The convergence for a fixed set is a consequence of (\ref{eqIntr8}). To prove the uniformity we note that if $\iota_i-\iota_{i-1}>\xi$, then $\iota_i^{(N)}-\iota_{i-1}^{(N)}>\xi/2$ as soon as $1/N<\xi/2$, which is independent of such $\iota$. This shows the claim for all such $\iota$ with two points and this concludes the proof because 
 $u_N(\iota)$ 
 is given by at most $\frac{1}{\xi}+1$-factors in this form.
\end{proof}

Another simple, but important, observation is that for a fixed $k\in \mathbb{N}$, with high probability the minimal distance between $Y_1^{(N)},\cdots,Y_k^{(N)}$ (the positions of the first $k$-maxima introduced in Section \ref{SecEEVE}) cannot be too small even if $N$ gets large. To be more precise, by using Lemma \ref{RSko}, we have that for any fixed $\varepsilon>0$ and $k\in\mathbb{N}$ there exist $\xi=\xi(k,\varepsilon)>0$ and ${N}_k$ such that for any $N>N_k$ the event
 \begin{equation}\label{eqDist1}
 \left\{\left|Y_i^{(N)}-Y_j^{(N)}\right|>{\xi},\, Y_{\ell}^{(N)}\in ({\xi}, 1-{\xi}),\, \forall\, \ell \textrm{, } i\neq j \in\{1,\cdots,k\}\right\}
 \end{equation}
 has probability larger than $1-\varepsilon$.
 By Proposition \ref{propoConc1} this implies that for any fixed $ \zeta>0$ on the event (\ref{eqDist1}), for all $N$ large enough and uniformly on $\iota=\{\iota_0=0<\iota_1<\cdots<\iota_{\ell}<1=\iota_{\ell+1}\}\subset Y^{(N,k)}$, cf. (\ref{RemStru11}), it holds that
\begin{align}\label{eqRF1}
 &e^{-N^{\gamma}E(\iota)-\zeta N^{\gamma}}\leq \mathrm{P}(\iota_1,\cdots,\iota_{\ell}\in \tau/N)\leq e^{-N^{\gamma}E(\iota)+\zeta N^{\gamma}},
\end{align}
where $E(\iota)=\sum_{i=1}^{\ell+1}(\iota_i-\iota_{i-1})^{\gamma}$ is the entropy of the set $\iota$, cf. (\ref{defentropyfiniteset}).

\begin{proof} [Proof of Lemma \ref{lemmaConc}]
The aim of this proof is to show that for any given $\delta>0$ and $k\in\mathbb{N}$ large enough, $\hat{\mathrm{P}}_{\hb_N,N}^{(k)}\left(d_H(\Idtr{\hb_N},I)>\delta\right)\to 0$ as $N\to\infty$, with an explicit rate of convergence. Our strategy is the following: given a set $I\subset\{0,1/N,\cdots, 1\}$, with $0,1\in I$, we consider 
\begin{equation}\label{eqHT:INKYNK}
I_{(N,k)}:=I\cap Y^{(N,k)},
\end{equation} the intersection of $I$ with the set of the positions of the first $k$-maxima: it can have distance larger or smaller than $\frac{\delta}{2}$ from $\Idtr{\hb_N}$. This induces a partition of the set of all possible $I$'s. This allows us to get the following inclusion of events
\begin{align}\label{eqEs1set}
&\left\{d_H(\Idtr{\hb_N},I)>\delta\right\}
\subset \left\{d_H(\Idtr{\hb_N},I_{(N,k)})\geq \frac{\delta}{2} \right\}\cup \left\{d_H(\Idtr{\hb_N}, I_{(N,k)})< \frac{\delta}{2}, d_H(I_{(N,k)},I)> \frac{\delta}{2} \right\}.
\end{align}
We have thus to prove our statement for 
 \begin{align}
 & \hat{\mathrm{P}}_{\hb_N,N}^{(k)}\left(d_H(\Idtr{\hb_N},I_{(N,k)})\geq \frac{\delta}{2}\right),\label{eqConc1}\\
  &\hat{\mathrm{P}}_{\hb_N,N}^{(k)}\left(d_H(\Idtr{\hb_N}, I_{(N,k)})< \frac{\delta}{2}, d_H(I,I_{(N,k)})> \frac{\delta}{2}\right).\label{eqConc2}
 \end{align}
 For this purpose we fix $\varepsilon>0$ and $\xi=\xi(\varepsilon,k)>0, \, N_k>0$ such that the event \eqref{eqDist1} holds with probability larger than $1-\varepsilon$, for any $N>N_k$. 
 
 \smallskip


Our goal is to find a good upper bound for (\ref{eqConc1}) and (\ref{eqConc2}). 
Let us start to consider (\ref{eqConc1}). Let $A$ be the set of all possible values of $I_{(N,k)}$, cf. \eqref{eqHT:INKYNK}, on the event $\{d_H(\Idtr{\hb_N},I_{(N,k)})\geq \frac{\delta}{2}\}$, namely
\begin{equation}
 A=\left\{\iota\subset Y^{(N,k)}: d_H(\iota,\Idtr{\hb_N})\geq \frac{\delta}{2} \textrm{ and } 0,1\in \iota\right\}.
\end{equation}
An upper bound of \eqref{eqConc1}) is
\begin{equation}
\hat{\mathrm{P}}_{\hb_N,N}^{(k)}\left(d_H(\Idtr{\hb_N},I_{(N,k)})\geq \frac{\delta}{2}\right)\, \leq\,  \displaystyle\sum_{\iota \in A} \hat{\mathrm{P}}_{\hb_N,N}^{(k)} (I_{(N,k)}=\iota).
\end{equation}

Let us fix $\zeta>0$ (we choose in a while its precise value) and assume that Relation \eqref{eqRF1} holds if $N_k$ is sufficiently large. Then
\begin{align}\label{eqTrkok}
 &\hat{\mathrm{P}}_{\hb_N,N}^{(k)} (I_{(N,k)}=\iota)=\frac{{\mathrm{E}}_{N}\left(e^{N^{\gamma} \hb_N\sDtr\left(I\right)}\mathds{1}({I_{(N,k)}=\iota}); 1\in \tau/N\right)}{{\mathrm{E}}_{N}\left( e^{N^{\gamma} \hb_N\sDtr\left(I\right)}; 1\in \tau/N\right)} 
 \leq  \frac{e^{ N^{\gamma} \hb_N\sDtr\left(\iota\right)}{\mathrm{P}}_N\left(\iota\subset I\right)}{e^{N^{\gamma} \hb_N\sDtr\left(\Idtr{\hb_N}\right)}{\mathrm{P}}_N\left(\Idtr{\hb_N}\subset I\right)}\overset{\eqref{eqRF1}}{\leq} 
  \\\nonumber
 \leq & \exp\left\{-N^{\gamma}(\udtr{\hb_N}-\Udtr{\hb_N}(\iota))+2N^{\gamma}\zeta\right\}\leq \exp\left\{-N^{\gamma}(\udtr{\hb_N}-\udtr{\hb_N}(\delta/2))+2N^{\gamma}\zeta\right\},
\end{align}
where $\Udtr{\hb_N}$ has been introduced in Definition \ref{EEdef1}.
By Proposition \ref{PropConU}, Part (\ref{con6}), if $k$ and $N_k$ are taken large enough, it holds that $\udtr{\hb_N}-\udtr{\hb_N}(\delta/2)>\eta$, for some $\eta>0$, with probability larger than $1-\varepsilon$.
We conclude that if $\zeta$ in (\ref{eqRF1}) is chosen smaller than $\eta/4$, then the l.h.s. of \eqref{eqTrkok} is bounded by $e^{-N^{\gamma}\frac{\eta}{2}}$, uniformly in $\iota\in A$.
By observing that $A$ has at most $2^k$ elements we conclude that
\begin{equation}
\eqref{eqConc1}\leq \displaystyle\sum_{\iota \in A} \hat{\mathrm{P}}_{\hb_N,N}^{(k)} (I_{(N,k)}=\iota)\leq |A|e^{-N^{\gamma}\eta/2}\leq 2^k e^{-N^{\gamma}\eta/2}.
\end{equation}

For (\ref{eqConc2}) we use the same strategy: Let $B$ be the set of all possible values of $I_{(N,k)}$, cf. \eqref{eqHT:INKYNK}, on the event $\{d_H(\Idtr{\hb_N}, I_{(N,k)})< \frac{\delta}{2}\}$,
\begin{equation}
 B=\left\{\iota\subset Y^{(N,k)}: d_H(\iota,\Idtr{\hb_N})< \frac{\delta}{2} \textrm{ and } 0,1\in \iota\right\}, 
\end{equation}
Then
\begin{equation}
\hat{\mathrm{P}}_{\hb_N,N}^{(k)}\left(d_H(\Idtr{\hb_N}, I_{(N,k)})< \frac{\delta}{2}, d_H(I,I_{(N,k)})> \frac{\delta}{2}\right)\,\leq\, \displaystyle\sum_{\iota \in B} \hat{\mathrm{P}}_{\hb_N,N}^{(k)}\left(d_H\left(\iota,I\right)>\frac{\delta}{2}, I_{(N,k)}=\iota\right).
\end{equation}
Let us observe that for such a given $\iota$
\begin{align}
&\hat{\mathrm{P}}_{\hb_N,N}^{(k)}\left(d_H\left(\iota,I\right)>\frac{\delta}{2} , I_{(N,k)}=\iota\right)=\frac{{\mathrm{E}}_{N}\left(e^{N^{\gamma} \hb_N\sDtr\left(I\right)}\mathds{1}({d_H\left(\iota,I\right)>\frac{\delta}{2}, I_{(N,k)}=\iota}); 1\in I\right)}{{\mathrm{E}}_{N}\left( e^{N^{\gamma} \hb_N\sDtr\left(I\right)}; 1\in I\right)} \\\nonumber
& \leq  \frac{{\mathrm{P}}_N\left(d_H\left(\iota,I\right)>\frac{\delta}{2}, I_{(N,k)}=\iota\right)}{{\mathrm{P}}_N\left(\iota\subset I\right)}.
\end{align}

We have reduced our problem to compute the probability of the event $\left\{d_H\left(\iota,I\right)>\frac{\delta}{2}, I_{(N,k)}=\iota\right\}$ under the original renewal distribution $\mathrm{P}_N$. 

Note that, if $\iota\subset I$, then $d_H\left(\iota,I\right)>\frac{\delta}{2}$ if and only if there exists $x\in I$ such that $d(x,\iota)>\frac{\delta}{2}$. Thus 
\begin{equation}\label{HT:setequaliy}
\left\{d_H\left(\iota,I\right)>\frac{\delta}{2}, I_{(N,k)}=\iota\right\}\, =\,\left\{\exists x\in I , d(x,\iota)>\frac{\delta}{2}, I_{(N,k)}=\iota\right\}.
\end{equation}

For $\iota=\{\iota_0=0<\iota_1<\cdots<\iota_{\ell}=1\}\in B$, we define $U_{j,\delta}:=[\iota_j+\frac{\delta}{2},\iota_{j+1}-\frac{\delta}{2}]\cap \frac{\mathbb{N}}{N}$, which is empty if the distance between $\iota_j$ and $\iota_{j+1}$ is strictly smaller than $\delta$. 
We can decompose the event \eqref{HT:setequaliy} by using such $U_{j,\delta}$, i.e., $ \{\exists x\in I , d(x,\iota)>\frac{\delta}{2}, I_{(N,k)}=\iota\}=\bigcup_{j=0}^{\ell-1}\bigcup_{x\in U_{j,\delta}}\{x\in I , I_{(N,k)}=\iota\}$, and we get
\begin{align}
 &{\mathrm{P}}_{N}\left(d_H\left(\iota,I\right)>\delta, I_{(N,k)}=\iota\right)
 \leq\displaystyle\sum_{j=0}^{\ell-1}\displaystyle\sum_{x\in U_{j,\delta}}{\mathrm{P}}_{N}\left(x\in I , I_{(N,k)}=\iota\right) \leq\displaystyle\sum_{j=0}^{\ell-1}\displaystyle\sum_{x\in U_{j,\delta}}{\mathrm{P}}_{N}\left(x\in I, \iota\subset I\right).
\end{align}

Let us consider ${\mathrm{P}}_{N}\left(x\in I, \iota\subset I\right)$. Since $x$ does not belong to $\iota$, there exists an index $j$ such that $\iota_j<x<\iota_{j+1}$. Then 
, recalling that $u(n)=\mathrm{P}(n\in\tau)$,
\begin{equation}\label{eqTRQPR}
\begin{split}
& \frac{{\mathrm{P}}_N\left(x\in I, \iota\subset I\right)}{{\mathrm{P}}_N\left(\iota\subset I\right)}\, =
\frac{\left[\prod\limits_{\substack{k=1,\\ k\neq j}}^{\ell-1}u(N(\iota_{k+1}-\iota_k))\right]\, u(N(x-\iota_j))u(N(\iota_{j+1}-x))}{\prod\limits_{k=1}^{\ell-1}u(N(\iota_{k+1}-\iota_k))}=
\frac{u(N(x-\iota_j))u(N(\iota_{j+1}-x))}{u(N(\iota_{j+1}-\iota_j))}\\
&\qquad \overset{\eqref{eqRF1}}{\leq } e^{-N^{\gamma}\Big((x-\iota_j)^{\gamma}+(\iota_{j+1}-x)^{\gamma}-(\iota_{j+1}-\iota_j)^{\gamma}\Big)+2\zeta N^{\gamma}}\leq e^{-N^{\gamma}(2^{1-\gamma}-1)\delta^{\gamma}+2\zeta N^{\gamma}},
 \end{split}
\end{equation}
uniformly on all such $\iota_j,\iota_{j+1}$ and $x$. Note that the last inequality follows by observing that for all such $\iota_j,\iota_{j+1}$ and $x$ one has $(x-\iota_j)^{\gamma}+(\iota_{j+1}-x)^{\gamma}-(\iota_{j+1}-\iota_j)^{\gamma}\geq (2^{1-\gamma}-1)\delta^{\gamma}$.
We conclude that, making possibly further restrictions on the value of $\zeta$ as function of $\delta$, there exists a constant $C>0$ such that $ \frac{{\mathrm{P}}_N\left(\{x\}\cup \iota\subset I\right)}{{\mathrm{P}}_N\left(\iota\subset I\right)}\leq e^{-CN^{\gamma}}$ uniformly in $\iota\in B$.
This leads to have that
\begin{equation}
  \eqref{eqConc2}\, \leq\,  \sum_{\iota\in B}\hat{\mathrm{P}}_{\hb_N,N}^{(k)}\left(d_H\left(\Idtr{\hb_N},I\right)>\frac{\delta}{2}, I_{(N,k)}=\iota\right)\leq|B| N e^{-C\, N^{\gamma}}\leq 2^k N e^{-C\, N^{\gamma}}.
\end{equation}
\end{proof}

\begin{proof}[Proof of Theorem \ref{TheoremMain2}]
First of all we are going to prove concentration around $\Idtr{\hb_N}$.
Let $k>0$ be fixed. Its precise value will be chosen in the following. Then, recalling Definition \ref{DefTruGibb},
\begin{align}
  &\Gib\left(d_H\left(\Idtr{\hb_N},I\right)>\delta \right)\leq  \\
  \leq & \hat{\mathrm{P}}_{\hb_N,N}^{(k)}\left(d_H\left(\Idtr{\hb_N},I\right)>\delta \right)\cdot \sup\left\{\frac{\mathrm{d}\Gib}{\mathrm{d}\hat{\mathrm{P}}_{\hb_N,N}^{(k)}}\left(I\right): d_H\left(\Idtr{\hb_N},I\right)>\delta \right\}.\nonumber
 \end{align}
To control the first term, by Lemma \ref{lemmaConc} for any $\varepsilon,\delta>0$ there exist $\nu>0$ and ${N}_k$ such that for all $N>{N}_k$, $\hat{\mathrm{P}}_{\hb_N,N}^{(k)}\left(d_H\left(\Idtr{\hb_N},I\right)>\delta \right)\leq e^{-N^{\gamma} \nu}$ with probability larger than $1-\varepsilon$ . To control the Radon-Nikodym derivative we may write
\begin{align}
& \frac{\mathrm{d}\Gib}{\mathrm{d}\hat{\mathrm{P}}_{\hb_N,N}^{(k)}}\left(I\right)=\frac{\hat{\mathrm{Z}}_{\hat{\beta}_N,N}^{(k)}}{\hat{\mathrm{Z}}_{\hat{\beta}_N,N}}\frac{e^{ N^{\gamma} \hb_N\sD\left(I\right)}}{e^{N^{\gamma} \hb_N\sDtr\left(I\right)}}\leq  e^{\hat{\beta}_N N^{\gamma} (\hb_N\sD(I)-\hb_N\sDtr\left(I\right))}\leq e^{\hat{\beta}_N N^{\gamma} \Rnk},
\end{align}
where $\Rnk=\sum_{i>k} M_i^{(N)}$ is defined in (\ref{eqRho1}).
By using Part (\ref{Tcnlemma3}) of Proposition \ref{PropConI} we choose $k$ large enough such that $\hat{\beta}_N\Rnk<\nu/2$ with probability $1-\varepsilon$, uniformly in $N$.
This forces to have
\begin{equation}
 \mathbb{P}\left(\Gib\left(d_H\left(\Idtr{\hb_N},I\right)>\delta \right)\leq e^{-N^{\gamma}\nu/2}\right)\geq 1-2\varepsilon.
\end{equation}
The proof follows by observing that if $k$ is large enough, then $d_H(\Idtr{\hb_N},\Id{\hb_N})<\delta/2$ with probability larger than $1-\varepsilon$, uniformly on $N$, cf. Point \eqref{con2} Proposition \ref{PropConI}.
\end{proof}
\subsection{Proof of Theorem \ref{TheoremMain3}}
In this section we prove Theorem \ref{TheoremMain3}.
The proof is based on the following result

\begin{lemma}\label{lemmagibbsdet}
 Let $(\mathcal{S},d_{\mathcal{S}})$ be a metric space and let $x_N$ be a sequence converging to $\bar{x}$. Let $\mu_N\in \mathcal{M}_1(\mathcal{S})$ be such that for any $\varepsilon>0$,
  $\lim_{N\to\infty} \mu_N\left(x:d(x_N,x)>\varepsilon\right)=0.$
Then
$
 \mu_N\rightharpoonup \delta_{\bar{x}}.
$
\end{lemma}

 \begin{proof} 
The proof is a consequence of the Portmanteau's Lemma \cite[Section 2]{B99}. 
\end{proof}

\begin{proof}[Proof of Theorem \ref{TheoremMain3}]
 Let $\mu_N=\hat{\mathrm{P}}_{\hat{\beta}_N,N}$ and $\mu_{\infty}=\delta_{\Ic{\hb}}$. Note that $\mu_N$ is a random measure on $\X$ depending on the discrete disorder $w^{(N)}$, while $\mu_{\infty}$ depends on the continuum disorder $w^{(\infty)}$. Therefore if we couple together these disorders as in Lemma \ref{RSko} we have that by Theorems \ref{TheoremMain2}, \ref{TheoremMain1} (see Remark \ref{RemOnTh1}) $\mu_{N}\left(I \mid d_H(\Id{\hb_N},I)>\delta\right)\overset{\mathbb{P}}{\to} 0$ and $\Id{\hb_N}\overset{\mathbb{P}}{\to} \Ic{\hb}$. 
To conclude the proof let us observe that the law of $\mu_N$ is a probability measure on $\mathcal{M}_1(\X)$, the space of the probability measures on $\X$, which is a compact space because $\X$ is compact. Therefore we can assume that $\mu_N$ has a limit in distribution. We have thus to show that this limit is the law of $\mu_{\infty}$. For this purpose it is enough to show that there exists a subsequence $N_k$ such that $\mu_{N_k}\tolaw\mu_{\infty}$. It is not difficult to check that we can find a subsequence $N_k$ such that $\mu_{N_k}\left(I \mid d_H(\hat{I}_{\hb_{N_k}, N_k},I)>\delta\right)\overset{\mathbb{P}-\textrm{a.s.}}{\to} 0$ and $\hat{I}_{\hb_{N_k}, N_k}\overset{\mathbb{P}-\textrm{a.s.}}{\to} \hat{I}_{\hb, \infty}$, therefore by Lemma \ref{lemmagibbsdet} we conclude that $\mu_{N_k}\overset{\mathbb{P}-\textrm{a.s.}}{\rightharpoonup}\mu_{\infty}$ and this concludes the proof.
\end{proof}

\section{Proof of Theorem \ref{TheoremMain4}}\label{StrI}
The goal of this section is to give a proof of Theorem \ref{TheoremMain4}.

\smallskip

As a preliminary fact let us show that if $\beta<\hbc$ then $\Ic{\beta}\equiv \{0,1\}$, while if $\beta>\hbc$ then $\Ic{\beta}\not\equiv \{0,1\}$.

\smallskip

 To this aim let us consider the maximum of the difference between the Continuum Energy (\ref{EqEnCont}) and the entropy (\ref{eqent2}), $\uc{\beta}=\sC(\Ic{\beta})-E(\Ic{\beta})$, defined in (\ref{eqVarEExyz}). Then whenever $\uc{\beta}\leq -1$, we have that $-1=-E(\{0,1\})\leq \uc{\beta}\leq -1$ and this implies that $\Ic{\beta}\equiv\{0,1\}$ by uniqueness of the maximizer.
On the other hand, if $\uc{\beta}>-1$, then there exists $I\not\equiv\{0,1\}$ such that $\Uc{\beta}(I)>-1$ because $\Uc{\beta}(\{0,1\})=-1$, so that $\{0,1\}\subsetneq \Ic{\beta}$. In particular, since $\beta\mapsto \uc{\beta}$ is non-decreasing, we have that $\Ic{\beta}\equiv\{0,1\}$ if $\beta < \hat{\beta}_{\texttt{c}}$ and $\Ic{\beta}\not\equiv\{0,1\}$ if $\beta > \hat{\beta}_{\texttt{c}}$.\\

To prove the theorem we proceed in two steps: in the first one we show that a.s. for any $\varepsilon>0$ there exists $\beta_0=\beta_0(\varepsilon)>0$ random for which $\Ic{\beta}\subset [0,\varepsilon]\cup [1-\varepsilon,1]$ for all $\beta<\beta_0$. In the second one we show that if $\varepsilon$ is small enough, then the quantity of energy that we can gain is always too small to hope to compensate the entropy. To improve this strategy we use some results on the Poisson Point Process that we are going to recall.\\

Let us start to note that the process $(Y_i^{(\infty)},M_i^{(\infty)})_{i\in\mathbb{N}}\subset [0,1]\times \mathbb{R}_+$ is a realization of a Poisson Point Process $\Pi$ with intensity 
\begin{equation}\label{eqmu}
 \mu(dxdz)=\mathbf{1}_{[0,1]}(x)\frac{\alpha}{z^{1+\alpha}}\mathbf{1}_{[0,\infty)}(z)dxdz.
\end{equation}
In such a way, as proved in \cite{K93}, the process
\begin{equation}
 X_{t}=\displaystyle\sum_{(x,z)\in\Pi} z\ind\left(x\in [0,t]\cup [1-t,1] \right), \quad t\in \left[0,\frac{1}{2}\right]
\end{equation}
is an $\alpha$-stable subordinator. The behavior of an $\alpha$-stable subordinator in a neighborhood of $0$ is described by \cite[Thm 10 Ch. 3]{B96}, precisely if $(X_t)_t$ is such subordinator with $\alpha\in(0,1)$ and $h:\mathbb{R}_+\to\mathbb{R}_+$ is an increasing function, then $\limsup_{t\to 0^+}X_t/h(t)=\infty\textrm{ or }0$ a.s. depending on whether the integral $\int_0^1 h(t)^{-\alpha}dt$ diverges or converges. In particular by taking $q>1$ and $h(t)=t^{1/\alpha}\log^{q/\alpha}(1/t)$ in a neighborhood of $0$, we have the following result
\begin{proposition}\label{propLPP}
Let $(X_t)_t$ be an  $\alpha$-stable subordinator, with $\alpha\in (0,1)$, then for every $q>1$ a.s. there exists a random constant $C>0$ such that
 \begin{equation}
  X_t\leq C t^{\frac{1}{\alpha}}\log^{\frac{q}{\alpha}}\left(\frac{1}{t}\right)
 \end{equation}
 in a neighborhood of $0$.
\end{proposition}

\begin{remark}
The process $X_t$ is the value of the sum of all charges in the set $[0,t]\cup [1-t,1]$. Therefore it gives an upper bound on the energy that we can gain by visiting this set.
\end{remark}
\subsubsection{Step One.}\label{prAI}
Let us show that a.s. for any $\varepsilon>0$ there exists $\beta_0=\beta_0(\varepsilon)>0$ for which $\Ic{\beta}\subset [0,\varepsilon]\cup [1-\varepsilon,1]$ for all $\beta<\beta_0$.
Otherwise there should exist $\varepsilon>0$ and a sequence $\beta_k>0$, $\beta_k\to 0$ as $k\to\infty$ such that $\Ic{\beta_k}\cap (\varepsilon, 1-\varepsilon)\neq \emptyset$. Let $x$ be one of such points, then, by Theorem \ref{ThmentropyCon1} we have that $E(\Ic{\beta_k})\geq E(\{0,x,1\})\geq  \varepsilon^{\gamma}+(1-\varepsilon)^{\gamma}$. Let $S=\sum_{i\in\mathbb{N}} M_i^{(\infty)}$, which is a.s. finite, cf. Remark \ref{rem:aboutMi}. Therefore by observing that $\uc{\beta_k}={\beta_k}\sC(\Ic{\beta_k})-E(\Ic{\beta_k})\geq -1$ we get
\begin{equation}
 \beta_k S\geq{{\beta_k}\sC(\Ic{\beta_k})}\geq {E(\Ic{\beta_k})-1}\geq {\varepsilon^{\gamma}+(1-\varepsilon)^{\gamma}-1}.
\end{equation}
There is a contradiction because the l.h.s. goes to $0$ as $\beta_k \to 0$, while the r.h.s. is a strictly positive number. 
\begin{remark}\label{RemBeta0}
 Let us note that if we set $\beta_0=\beta_0(\varepsilon)=(\varepsilon^{\gamma}+(1-\varepsilon)^{\gamma}-1)/S$ then for all $\beta<\beta_0$ it must be that $\Ic{\beta}\subset [0,\varepsilon]\cup [1-\varepsilon,1]$. Moreover 
 $\beta_0\downarrow 0$ as $\varepsilon\downarrow 0$.
\end{remark}
\subsubsection{Step Two.}
Now let us fix $\varepsilon>0$ small and $\beta_0=\beta_0(\varepsilon)\leq 1$ as in Remark \ref{RemBeta0}. 
Let
\begin{align}
 &\varepsilon_1=\sup \Ic{\beta}\cap [0,\varepsilon], \\
 &\varepsilon_2=\inf \Ic{\beta}\cap [1-\varepsilon,1].
\end{align}
Let $\hat{\varepsilon}=\max\{\varepsilon_1,1-\varepsilon_2\}$. If $\hat{\varepsilon}=0$ we have finished. Then we may assume that $\hat{\varepsilon}>0$ and we choose $q>1$, $C>0$ for which Proposition \ref{propLPP} holds for any $t<\varepsilon$, namely 
\begin{align}
{\beta}\sC(\Ic{\beta})\leq \beta_0 X_{\hat{\varepsilon}}\leq  C {\hat{\varepsilon}}^{\frac{1}{\alpha}}\log^{\frac{q}{\alpha}}\left(\frac{1}{\hat{\varepsilon}}\right),
\end{align}

By Theorem \ref{ThmentropyCon1} we get a lower bound for the entropy
\begin{equation}
 E(\Ic{\beta})\geq \hat{\varepsilon}^{\gamma}+(1-\hat{\varepsilon})^{\gamma}.
\end{equation}
In particular if $\varepsilon$ is small enough, we have that 
\begin{equation}
 E(\Ic{\beta})-1> \frac{\hat{\varepsilon}^{\gamma}}{2}.
\end{equation}

Therefore with a further restrictions on $\varepsilon$ and $\beta_0$, if necessary, by recalling that $\alpha,\gamma\in (0,1)$ we conclude that for all $\beta<\beta_0$
\begin{align}
E(\Ic{\beta})-1\leq {\beta}\sC(\Ic{\beta})\leq    C {\hat{\varepsilon}}^{\frac{1}{\alpha}}\log^{\frac{q}{\alpha}}\left(\frac{1}{\hat{\varepsilon}}\right)\leq \frac{\hat{\varepsilon}^{\gamma}}{2}< E(\Ic{\beta})-1,
\end{align}
which is a contradiction. Therefore $\hat{\varepsilon}$ must be $0$ and this implies that $\Ic{\beta}\equiv \{0,1\}$ for each $\beta<\beta_0$.

\section{The Directed polymer in random environment with heavy tails}\label{SecGamma}
Originally introduced by \cite{HH85}, the directed polymer in random environment is a model to describe an interaction between a polymer chain and a medium with microscopic impurities. From a mathematical point of view we consider the set of all possible paths of a $1+1$ - dimensional simple random walk starting from $0$ and constrained to come back to $0$ after $N$-steps. The impurities --- and so the medium-polymer interactions --- are idealized by an i.i.d. sequence $(\{\omega_{i,j}\}_{i\in\mathbb{N}, j\in\mathbb{Z}},\mathbb{P})$. Each random variable $\omega_{i,j}$ is placed on the point $(i,j)\in \mathbb{N}\times\mathbb{Z}$. 
%
For a given path $s$ we define the Gibbs measure
\begin{equation}
 \mu_{\beta,N}(s)=\frac{e^{\beta\sigma_N(s)}}{Q_{\beta,N}},
\end{equation}
where $\sigma_N(\cdot)=\sum_{i,j} \omega_{i,j} \ind(s_i=j)$ is the energy and $Q_{\beta,N}$ is a normalization constant.

In \cite{AL11} is studied the case in which the impurities have heavy tails, namely the distribution of $\omega_{1,1}$ is regularly varying with index $\alpha\in (0,2)$. 
In this case to have a non-trivial limit as $N\to\infty$, we have to choose $\beta=\beta_N\sim \hb N^{1-2/\alpha}L(N)$, with $L$ a slowly varying function, cf. \cite[(2.4),(2.5)]{AL11}. 
For such a choice of $\beta$, cf. \cite[Theorem 2.1]{AL11}, one has that the trajectories of the polymer are concentrated in the uniform topology around a favorable curve $\hat{\gamma}_{\beta_N,N}$. In \cite[Theorem 2.2]{AL11} one shows that there exists a limit in distribution for the sequence of curves $\hat{\gamma}_{\beta_N,N}$, denoted by $\hat{\gamma}_{\beta}$. Moreover there exists a random threshold $\beta_c$ below which such limit is trivial ($\hat{\gamma}_{\beta}\equiv 0$), cf. \cite[Proposition 2.5]{AL11}. 
Anyway a complete description of $\beta_c$ was not given, see Remark \ref{rem:AL11bcora}. In our work we solve this problem, cf. Theorem \ref{eq32al}. 
\medskip

The rest of the section is consecrated to prove Theorem \ref{eq32al}.

\begin{definition}[entropy]\label{defEntrA} 
Let us consider $\mathcal{L}^0=\{s:[0,1]\to \mathbb{R} : s \textrm{ is }1-\textrm{Lipschitz}, s(0)=s(1)=0\}$ equipped with $L^{\infty}$-norm, denoted by $\|\cdot\|_{\infty}$.

For a curve $\gamma\in \mathcal{L}^0$ we define its entropy as
\begin{equation}\label{entropy}
E(\gamma)=\int_0^1 e\left(\frac{d}{dx}\gamma(x)\right) dx,
\end{equation}
where $e(x)=\frac{1}{2}((1+x)\log(1+x)+(1-x)\log(1-x))$.
\end{definition}

Let us observe that $E(\cdot)$ is the rate function in the large deviations principle for the sequence of uniform measures on $\mathcal{L}_N^0$, the set of linearly interpolated $\frac{1}{N}$-scaled trajectories of a simple random walk.
\begin{definition}
We introduce the continuous environment $\pi_{\infty}$ as
\begin{equation}\label{eqContBETA}
\pi_{\infty}(\gamma)=\sum_i T_i^{-\frac{1}{\alpha}}\delta_{Z_i}\left(\textrm{graph}(\gamma)\right), \gamma\in \mathcal{L}^0.
\end{equation}
Here $\textrm{graph}(\gamma)=\{(x,\gamma(x)) : x\in [0,1])\}\subset \mathcal{D}:=\{(x,y)\subset\mathbb{R}^2: |y|\leq x\wedge (1-x)\}$ is the graph of $\gamma$,  $\alpha\in (0,2)$ is the parameter related to the disorder, $T_i$ is a sum of $i$-independent exponentials of mean $1$ and $(Z_i)_{i\in\mathbb{N}}$ is an i.i.d.-sequence of $\textrm{Uniform}(\mathcal{D})$ r.v.'s. These two sequences are assumed to be independent with joint law denoted by $\mathbb{P}_{\infty}$.
\end{definition}

For $\beta<\infty$ we introduce
\begin{equation}
 \hat{\gamma}_{\beta}=\displaystyle\argmax_{\gamma\in\mathcal{L}^0}\left\{\, \beta\pi_{\infty}(\gamma)-E(\gamma) \, \right\}
\end{equation}
and we set $u_{\beta}=\beta\pi_{\infty}(\hat{\gamma}_{\beta})-E(\hat{\gamma}_{\beta}) $. Since $\beta\pi_{\infty}(\gamma\equiv 0)-E(\gamma\equiv 0)=0$ a.s. we have that  $u_{\beta}\geq 0$ a.s., consequently we define the random threshold as
\begin{equation}\label{ubetacAA}
 \beta_c=\inf\{\beta>0 : u_{\beta}>0\}=\inf\{\beta>0 : \hat{\gamma}_{\beta}\not\equiv 0\}.
\end{equation}

\subsection{The Structure of $\beta_c$}
The random set $(Z_i,T_i^{-\frac{1}{\alpha}})_{i\in\mathbb{N}}\subset \mathcal{D}\times \mathbb{R}_+$ is a realization of a Poisson Point Process 
, denoted by $\Pi^*$, with density given by
\begin{equation}\label{eqmu}
 \mu^*(dxdydz)=\frac{\mathbf{1}_{\mathcal{D}}(x,y)}{|\mathcal{D}|}\frac{\alpha}{z^{1+\alpha}}\mathbf{1}_{[0,\infty)}(z)dxdydz.
\end{equation}
Let us introduce the process 
\begin{equation}\label{eq:UTRMPS}
 U_{t}=\displaystyle\sum_{(x,y,z)\in\Pi^*}z\ind\left( (x,y)\in A(t) \right), \quad t\in \left[0,\frac{1}{2}\right]
\end{equation}
with $A(t)=\{(x,y)\in\mathcal{D}: x\in [0,1], |y| \leq t\}$. 
Let us observe that the process $(U_t)_{t\in [0,\frac{1}{2}]}$  is "almost" a Lévy Process, 
in sense that it has càdlàg trajectories and independent but not homogeneous increments because the area of $A(t)$ does not grow linearly.
Anyway, by introducing a suitable function $\phi(t)>t$, we can replace $A(t)$ by $A(\phi(t))$  to obtain a process with homogeneous increment. In particular we take $\phi(t)=1/2(1-\sqrt{1-4t})$ in order to have that $\textrm{Leb}(A(\phi(t)))=t$ for all $t\in [0,1/4]$. Then the process
\begin{equation}\label{eq:UTRMPS1}
W_{t}=U_{\phi(t)}
\end{equation}
is a subordinator and $W_t\geq U_t$ for any $t\in [0,1/4]$. 

\medskip

Before giving the proof of Theorem \ref{eq32al}, we prove a general property of the model: 

\begin{proposition}\label{prop:bc1osti}
 For any fixed $\alpha\in (0,2)$, $\mathbb{P}_{\infty}$-a.s. for any $\varepsilon>0$ there exists $\beta_0=\beta_0(\varepsilon)>0$ such that $\|\hat{\gamma}_{\beta}\|_{\infty}<\varepsilon$  (that is, $\textrm{graph}(\hat{\gamma}_{\beta})\subset A({\varepsilon})$) for all $\beta<\beta_0$.
\end{proposition}

Let us recall some preliminary results necessary for the proof.

\begin{proposition}\label{ProentropyA}
 Let $E$ be the entropy of Definition \ref{defEntrA}. Then for all $\gamma\in \mathcal{L}^0$ if $z=(x,y)\in \textrm{graph}(\hat{\gamma}_{\beta})$ we have that 
 \begin{equation}
  E(\gamma)\geq E(\gamma_{z}),
 \end{equation}
where $\gamma_z$ is the curve obtained by linear interpolation of $\{(0,0), z, (1,0)\}$.
\end{proposition}
\begin{proof}
 \cite[Proposition 3.1]{AL11}
\end{proof}

As shown in \cite[Proof of Proposition 2.5]{AL11}, there exist two constants $C_1,C_2>0$  such that for all $z=(x,y)\in \mathcal{D}$ we have $C_1\left(\frac{y^2}{x}+\frac{y^2}{1-x}\right)\leq E(\gamma_z)\leq C_2\left(\frac{y^2}{x}+\frac{y^2}{1-x}\right)$. This implies that there exists $C_0>0$ for which
\begin{equation}\label{eqSquare}
 E(\gamma_z)\geq C_0 y^2,
\end{equation}
uniformly on $z\in\mathcal{D}$.

\begin{proof}[Proof of Proposition \ref{prop:bc1osti}] By contradiction let us suppose that there exists $\varepsilon>0$ such that for a sequence $\beta_k\rightarrow 0$ as $k\to\infty$ we have $\|\hat{\gamma}_{\beta_k}\|_{\infty}\geq\varepsilon$. By continuity of $\hat{\gamma}_{\beta_k}$ there exists a point $x\in [\varepsilon,1-\varepsilon]$ such that $\hat{\gamma}_{\beta_k}(x)=\varepsilon$.
By \cite[Proposition 4.1]{HM07}, with probability $1$ there exists a random set $A\subset \mathbb{N}$ such that $S=\sum_{i\in A}T_i^{-\frac{1}{\alpha}}<\infty$ and for any $\gamma\in \mathcal{L}^0$ it holds that $S\geq \pi_{\infty}(\gamma)$. For instance if $\alpha\in (0,1)$, then we can choose $A\equiv \mathbb{N}$, while if $\alpha>1$, then $A\subsetneq\mathbb{N}$. Since $u_{\beta_k}=\beta_k \pi_{\infty}(\hat{\gamma}_{\beta_k})-E(\hat{\gamma}_{\beta_k})\geq 0$ we obtain that a.s.
\begin{equation}
 \beta_k S\geq \beta_k \pi_{\infty}(\hat{\gamma}_{\beta_k})\geq {E(\hat{\gamma}_{\beta_k})}\geq {E({\gamma}_{z=(x,\varepsilon)})} \geq  {C_0\varepsilon^2}.
\end{equation}
Sending $\beta_k\to 0$ we obtain a contradiction because the l.h.s. converges to $0$.
\end{proof}

We are now ready to prove Theorem \ref{eq32al}.

\begin{proof}[Proof of Theorem \ref{eq32al}] We have to prove only the point $(1)$, the other one has been already proven in \cite{AL11}.
Let $\varepsilon>0$ be fixed and $\beta_0=\beta_0(\varepsilon)$ such that $\|\hat{\gamma}_{\beta}\|_{\infty}<\varepsilon$ for all $\beta<\beta_0\leq 1$.
Moreover we define $\hat{\varepsilon}:=\max|\hat{\gamma}_{\beta}(x)|$. 

\smallskip

An upper bound for the energy gained by $\hat{\gamma}_{\beta}$ is given by $\sum_{i\in\mathbb{N}}T_i^{-\frac{1}{\alpha}}\textbf{1}_{(Z_i\in A_{\hat{\varepsilon}})}$, the sum of all charges contained in the region $A_{\hat{\varepsilon}}$. Such quantity is estimated by the process $W_{\hat{\varepsilon}}$, cf. \eqref{eq:UTRMPS1}. Therefore by Proposition \ref{propLPP} we can choose suitable constants $q>1$ and $C>0$ such that

\begin{equation}
 \pi_{\infty}(\hat{\gamma}_{\beta})\leq  \displaystyle\sum_{i\in\mathbb{N}}T_i^{-\frac{1}{\alpha}}\textbf{1}_{(Z_i\in A_{\hat{\varepsilon}})}\leq U_{\hat{\varepsilon}} \leq C \hat{\varepsilon}^{\frac{1}{\alpha}}\log^{\frac{q}{\alpha}}\left(\frac{1}{\hat{\varepsilon}}\right).
\end{equation}

A lower bound for the entropy is provided by \eqref{eqSquare}:
\begin{equation}
 E(\hat{\gamma_{\beta}})\geq C_0 \hat{\varepsilon}^2.
\end{equation}
Conclusion: if $\varepsilon$ is small enough we get
\begin{equation}
 \beta \pi_{\infty}(\hat{\gamma}_{\beta})\leq C \hat{\varepsilon}^{\frac{1}{\alpha}}\log^{\frac{q}{\alpha}}\left(\frac{1}{\hat{\varepsilon}}\right) \leq C_0 \hat{\varepsilon}^2\leq  E(\hat{\gamma_{\beta}}),
\end{equation}
because $\alpha <\frac{1}{2}$ and this forces $u_{\beta}=0$ for all $\beta<\beta_0$.
\end{proof}

\section{Possible Generalizations and Perspective}\label{SecPersp}

This work represents the first analysis of such model with this particular choice of the disorder and renewal process. There are several open questions regarding mainly, but not only, the comprehension of the model with different choices of renewal process: 

\begin{itemize}

\item ($\gamma\geq 1$) The condition $\gamma\geq 1$ implies that the entropy function $E(I)$, cf. \eqref{eqResc}, is non-increasing (strictly non-increasing if $\gamma>1$) with respect to the inclusion of sets in $\Xdi$, cf. \eqref{eqDensity1}. It turns out that for any fixed $\beta>0$ and $N\in\mathbb{N}$, the solution of \eqref{eqIntr5} is $\Idp{\beta}=\{0,1/N,\cdots,1\}$. Therefore, whenever $N\to \infty$, the limit set is given by the interval $[0,1]$, independently of our choice of $\beta_N$. We conjecture that $\tau/N$ converges to the whole segment $[0,1]$.

\smallskip
  
\item ($\gamma=0$) The case $\gamma=0$ corresponds to consider a renewal process with polynomial tail, that is $K(n):=\mathrm{P}(\tau_1=n)\sim L(n)n^{-\rho}$, with $\rho>1$, cf. \eqref{eqIntr6}. 
In this case we conjecture that the correct rescaling is given by $\beta=\beta_N\sim N^{-1/\alpha}\log N$ and the limit measure for the the sequence $\mathrm{P}_{{\beta}_N,h,N}^{{\omega}}(\cdot)$ is given by a more complicated structure than the $\delta$-measure of a single set. This would mean that we do not have concentration around a single favorable set.

\smallskip

\item An interesting open problem is given by the structure of $\Ic{\hb}$, cf. \eqref{eqVarEE} and Theorem \ref{TheoremMain1}. In Theorem \ref{TheoremMain4} we have proven that if $\hb$ is small enough, then $\Ic{\hb}\equiv \{0,1\}$ a.s., otherwise, if $\hb$ is large, $\{0,1\}\subsetneq \Ic{\hb}$. We conjecture that for any finite $\hb>0$ it is given by a finite number of points.


\end{itemize}

\appendix
\section{Asymptotic Behavior for Terminating Renewal Processes}\label{Aa1}
In this section we consider a terminating renewal process $(\tau,\mathrm{P})$ and $K(n)=\mathrm{P}(\tau_1=n)$, with $K(\infty)>0$. The aim is to study the asymptotic behavior of the renewal function $u(N)=\mathrm{P}(N\in\tau)=\sum_m K^{*(m)}(N)$, where $K^{*(m)}$ is the $m^{th}$-convolution of $K$ with itself, under the assumption that $K(\cdot)$ is subexponential. We refer to \cite{FKZ09} for the general theory of the subexponential distributions
.


\begin{definition}[Subexponential distribution]\label{DefSubExp}
We say that a discrete probability density $q$ on $\mathbb{N}$ is subexponential if
 \begin{equation}
 \forall\, k>0,\, \lim_{n\to\infty}q(n+k)/q(n)=1 \quad \text{and}\quad \lim_{n\to\infty}q^{*(2)}(n)/q(n)=2,
 \end{equation}
\end{definition}
The result we are interested in is the following 
\begin{theorem}\label{ThmSD}
 Let $K(\cdot)$ be a discrete probability density on $\mathbb{N}\cup\{\infty\}$ such that $K(\infty)>0$ and let $\delta=1-K(\infty)<1$. Let $q(\cdot)$ be defined as $q(n)=\delta^{-1}K(n)$. If $q$ is subexponential, then 
 \begin{equation}
  \displaystyle\lim_{n\to\infty} \frac{u(n)}{K(n)}=\frac{1}{K(\infty)^2}.
 \end{equation}
\end{theorem}
Its proof is a simple consequence of the Dominated Convergence Theorem by using the following results
\begin{lemma}\label{Cor1}
 Let $q$ be a subexponential discrete probability density on $\mathbb{N}$, then for any $m\geq 1$
 \begin{equation}
  q^{*(m)}(n)\overset{n\to\infty}{\sim} m q(n).
 \end{equation}
\end{lemma}
\begin{proof}
 \cite[Corollary 4.13]{FKZ09}.
\end{proof}
\begin{theorem}\label{DomThm}
Let $q$ be a subexponential discrete probability density on $\mathbb{N}$. Then we have that
 for any $\varepsilon>0$ there exist $N_0=N_0(\varepsilon)$ and $c=c(\varepsilon)$ such that for any $n>N_0$ and
 $m\geq 1$
 \begin{equation}
  q^{*(m)}(n)\leq c (1+\varepsilon)^m q(n).
 \end{equation}
\end{theorem}
\begin{proof}
 \cite[Theorem 4.14]{FKZ09}.
\end{proof}
\subsection{The case of $K(n)\cong  e^{-\mathtt{c}n^{\gamma}}$.} \label{A1.2}
In this section we want to show that \eqref{eqIntr6} satisfies Assumption \ref{ASSrenP}. The fact that it is stretched-exponential, \eqref{A2}, is obvious, then 
it is left to prove that it is subexponential, \eqref{A1}.
%
%

\smallskip

By \cite[Theorem 4.11]{FKZ09}, we can assume $K(n)=  n^{\rho} \tilde{L}(n) e^{-\mathtt{c}n^{\gamma}}$, where $\tilde{L}$ is another slowly varying function such that $\tilde{L}(n)\sim L(n)$ as $n\to\infty$.
Since $\gamma\in (0,1)$ 
we get that
 for any fixed $k>0$, $\lim_{n\to\infty} {K(n+k)}/{K(n)}=1$. Such property goes under the name of long-tailed and it allows to apply \cite[Theorem 4.7]{FKZ09}: to prove that $K$ is subexponential, we have to prove that for any choice of $h=h(n)\to \infty$ as $n\to \infty$, with $h(n)<n/2$, we have that $\sum_{m=h(n)}^{n-h(n)}K(n-m)K(m)=o(K(n)),$ as $n\to \infty$.
Let us consider $R(y)=y^{\gamma}$, with $\gamma\in (0,1)$. $R$ is a concave increasing function and $R'(y)=\gamma y^{\gamma-1}$ is strictly decreasing, so that given two integer points $n,m$ such that $n-m>m$ we have
\begin{align}
 R(n) - R(n-m)\leq m R'(n-m)\leq mR'(m)=\gamma m^{\gamma}=\gamma R(m),
\end{align}
By Karamata's representation for slowly varying functions \cite[Theorem 1.2.1]{BGT89} there exists $c_1\geq 1$ for which $\tilde{L}(xr)\leq c_1\tilde{L}(r)$ for any $x\in [\frac{1}{2},1]$ and $r\geq 1$. This implies also that for any $\rho\in\mathbb{R}$ there exists $c=c(\rho)$ such that $(xr)^{\rho}\tilde{L}(xr)\leq cr^{\rho}\tilde{L}(r)$ for any $x\in [\frac{1}{2},1]$ and $r\geq 1$.
Therefore in our case, whenever $n-m\geq n/2$ we have that $K(n-m)\leq n^{\rho}\tilde{L}(n)e^{-\mathtt{c}(n-m)^{\gamma}}=K(n)e^{R(n)-R(n-m)}$. 
Summarizing, by using all these observations we conclude that 
\begin{equation}
\sum_{m=h(n)}^{\frac{n}{2}}\frac{K(n-m)K(m)}{K(n)}\leq c\sum_{m=h(n)}^{\infty} m^{\rho}\tilde{L}(m)e^{-\mathtt{c}(1-\gamma)R(m)},
\end{equation} 
which goes to $0$ as $h(n)\to \infty$ and the proof follows by observing that
\begin{equation}
\sum_{m=h(n)}^{n-h(n)}\frac{K(n-m)K(m)}{K(n)}=2\sum_{m=h(n)}^{\frac{n}{2}}\frac{K(n-m)K(m)}{K(n)}.
\end{equation}
 $\hfill \Box$

\section*{Acknowledgments}
I wish to thank Francesco Caravenna and Fabio Lucio Toninelli 
for their constant support to develop this work.
%
%


\bibliographystyle{amsplain}
\bibliography{biblioT}

\end{document}